\documentclass[a4paper,12pt]{article}

\usepackage[T1]{fontenc}  
\usepackage[latin1]{inputenc}  
\usepackage{amsthm,amsmath,stmaryrd,bbm,hyperref,geometry,color}
\usepackage{amssymb}
\usepackage[english]{babel}
\usepackage{graphicx}
\usepackage{amsfonts,amssymb}
\usepackage{verbatim}
\usepackage{enumitem}
\usepackage[all]{xy}

\newcommand{\pierre}[1]{{\color{black}#1}}
\newcommand{\lj}[1]{{\color{black}#1}} 
\newcommand{\po}{\left(}
\newcommand{\pf}{\right)}

\newcommand{\E}{\mathbb E}
\newcommand{\R}{\mathbb R} 
 
\newcommand{\C}{\mathcal C}
\newcommand{\Z}{\mathbb Z} 
 
\newcommand{\N}{\mathbb N} 
\newcommand{\B}{\mathcal B}
\newcommand{\F}{\mathcal F}
\newcommand{\dd}{\text{d}}
\newcommand{\na}{\nabla}

\newcommand{\proc}[2]{\mathbb{P}\left(#1~\middle|~#2\right)}
\newcommand{\procz}[3]{\mathbb{P}_{#1}\left(#2~\middle|~#3\right)}
\newcommand{\ec}[2]{\mathbb{E}\left(#1~\middle|~#2\right)}

\newtheorem{theorem}{Theorem}
\newtheorem{assumption}{Assumption}
\newtheorem{lemma}[theorem]{Lemma}

\newtheorem{proposition}[theorem]{Proposition}
\newtheorem{corollary}[theorem]{Corollary}
\newtheorem{remark}{Remark}

\title{Convergence of the kinetic annealing for general potentials}
\author{Lucas Journel, Pierre Monmarch\'e}

\begin{document}

\maketitle

\abstract{The convergence of the kinetic Langevin simulated annealing is proven under mild assumptions on the potential $U$ for slow logarithmic cooling schedules, which widely extends the scope of the previous results of \cite{Pierro}. Moreover, non-convergence for  fast logarithmic and non-logarithmic cooling schedules is established. The results are based on an adaptation to  non-elliptic non-reversible kinetic settings of a localization/local convergence strategy developed by Fournier and Tardif in \cite{FourTar} in the overdamped elliptic case, and on precise quantitative high order Sobolev hypocoercive estimates. }

\medskip

\noindent\textbf{Keywords:} Langevin diffusion ; simulated annealing ; hypocoercivity ; metastability ; stochastic optimization

\medskip

\noindent\textbf{MSC class:} 60J60 ; 46N30

\section{Introduction and main results}

Given a potential $U:\R^d\mapsto\R$, the goal of a simulated annealing procedure is to minimize  $U$ by designing a stochastic process $(X_t)_{t\geqslant0}$ whose law at time $t$ is close to the probability density proportional to $e^{-\beta_t U(x)}\dd x,$  
 where $\beta:\R_+\mapsto\R_+$ is called the cooling schedule. As $\beta$ goes to infinity, this probability law concentrates around the global minimizers of $U$.

The most classical case is based on the  overdamped Langevin process:
\[ \dd X_t = -\beta_t\na U(X_t)\dd t + \sqrt{2}\dd B_t\,, \]
where $(B_t)_{t\geqslant 0}$ is a standard Brownian motion on $\R^d$. For a fixed $\beta$, the density $e^{-\beta U(x)}\dd x$ is a stationary measure for this process. The idea is thus that, if $\beta$ increases sufficiently slowly, the law of $X_t$ gets and remains close to its instantaneous equilibrium. As a consequence, the convergence of the simulated annealing algorithm, in the sense of convergence in probability of $U(X_t)$ toward $\min U$ as $t\rightarrow +\infty$, is related to the longtime convergence to equilibrium of the process at a fixed but high $\beta$. On the contrary, when $\beta$ goes to infinity too fast, the algorithm is expected to fail with positive probability, i.e. the law of $X_t$ is not expected to be close to $e^{-\beta_t U}$ and $U(X_t)$ to converge to $\min U$.

Proof of the convergence of the overdamped Langevin simulated annealing for slow logarithmic cooling schedule (with the optimal condition involving the critical height of the potential, see below) and of the non-convergence for fast logarithmic cooling schedule  was first established by Holley, Kusuoaka and Strook \cite{HolStr,HoStKu}, using Sobolev inequalities, for a potential $U$ on a compact manifold. The case of $\R^n$ has been studied by Chiang, Hwang and Sheu \cite{CHS}, Royer \cite{Royer} and Miclo \cite{Miclo}, under restrictive conditions on the behavior at infinity of $U$, in particular $|\na U| \rightarrow +\infty$ at infinity. These conditions are related to the functional inequalities used in these works, in particular spectral gap and Nelson hypercontractivity inequalities. The question of reducing these assumptions in order to consider slowly-growing potentials has been addressed by Zitt in \cite{Zitt}, essentially by replacing spectral gap inequalities by weaker functional inequalities. The results of Zitt apply for instance if, outside some ball, $U(x)=|x|^\alpha$  with $\alpha\in(0,1)$. 

More recently, Fournier and Tardif  in \cite{FourTar} and with one of the author in \cite{FTM} have been interested in somehow minimal conditions on the growth of $U$ at infinity. In \cite{FTM}, they established that, for coercive potentials in the sense that $x\cdot \na U(x) \geqslant 0$ for $|x|$ large enough, there is a phase transition for $U_\alpha(x) = \alpha \ln(1+\ln(1+|x|^2))$ at some value $\alpha_*$ of $\alpha$ (depending on  the cooling schedule  $\beta$ and the dimension), i.e. there is convergence if $\alpha>\alpha_*$ and non-convergence if $\alpha<\alpha_*$, which is related to the transient properties of Bessel processes. More generally, convergence of the annealing algorithm is also proven in \cite{FTM} under conditions that allow arbitrarily slow growth. In \cite{FourTar},  the convergence of the simulated annealing is established as soon as $\lim_{|x|\rightarrow\infty}U(x)=\infty$ and $\int_{\R^d}e^{-\alpha_0U(x)}\dd x<\infty$ for some $\alpha_0>0$. Although it doesn't cover all the cases of \cite{FTM} (notice indeed that the condition is not met for $U_\alpha$ for any $\alpha>0$), this is a very simple and mild condition. One of the main differences of \cite{FourTar,FTM} with respect to previous works is that the question of the recurrence of the process is treated separately from the question of convergence in probability to the minimum of $U$. Indeed, once recurrence is proven, it is essentially sufficient to use the known results of convergence in the compact case to conclude.   Notice that, unfortunately,  this localization argument does not provide a rate of convergence as did previous works. \lj{Note as well} that the idea that the behavior of $U$ at infinity is not so important already appears in \cite{CHS,Royer} (see in particular \cite[Lemma 6.4]{CHS}) where  it is proven that it is sufficient (under the conditions enforced in these works) to prove the result in the case where $U(x)=|x|^4$ for $|x|$ large enough.


On the other hand, the second author studied in \cite{Pierro} the simulated annealing based on the kinetic Langevin process:
\begin{equation*}
  \left\{
      \begin{aligned}
        &\dd X_t = Y_t\dd t\\    
        &\dd Y_t = -\nabla U(X_t)\dd t - \beta_t Y_t\dd t + \sqrt{2}\dd B_t\,.\\
      \end{aligned}
    \right.  
\end{equation*}
For a constant $ \beta$, the invariant measure is proportional to $e^{-\beta H}$ with the Hamiltonian $H(x,y)= U(x)+|y|^2/2$. The use of this process, which is non-reversible and has a ballistic rather than diffusive behavior, is motivated by its better convergence properties with respect to the overdamped process (although, as discussed in \cite{Pierro}, in the regime $\beta\rightarrow +\infty$, it doesn't reduce the critical height of the energy landscape). Convergence of the kinetic simulated annealing in established in \cite{Pierro} for slow logarithmic cooling schedules similar to the overdamped case, under restrictive conditions on $U$, namely $U$ is essentially quadratic at infinity, and the Hessian of $U$ is bounded. The proof is  similar to the overdamped Langevin case except that establishing quantitative longtime convergence estimates (at a fixed $\beta$) for the process toward its equilibrium rely on so-called hypocoercive methods, as introduced by Villani in \cite{Vil}. The arguments of \cite{Pierro} have been adapted to     Generalized Langevin processes  in \cite{GenLan}.



The present work is concerned with the kinetic Langevin simulated annealing. Our contributions with respect to   \cite{Pierro} are the following. First, following the method of \cite{FourTar}, the conditions on $U$ are considerably weakened. Notice that, in the kinetic case, it means we can consider potentials that grow slower than those considered in \cite{Pierro}, but also potentials that grow much faster (arbitrarily fast in fact), while the results of \cite{Pierro} require $U(x)/(|x|^2+1)$  and $\na^2 U$ to be bounded. Second, we prove the failure of the algorithm with fast cooling schedule, which was yet to be established in an hypocoercive case. Indeed, the failure of the overdamped Langevin simulated annealing is proven in \cite{HoStKu} thanks to hypercontractivity results that are not available for the kinetic process, and the hypocoercive convergence results used in \cite{Pierro} to prove convergence are too weak to conclude \lj{about the failure of the algorithm} in the fast cooling case \pierre{(see the discussion at the beginning of Section~\ref{fullprocess})}. By proving the non-convergence of the algorithm under the same condition as in the overdamped case, we make rigorous the heuristic discussion in \cite{Pierro} according to which the kinetic process does not change the optimal condition on the cooling schedule. Third, as noted in \cite{GenLan},  a technical truncation argument in \cite{Pierro}, required for the rigorous computation of the modified entropy dissipation, is incorrect, and  we have solved this issue (see Section~\ref{hypo1} \pierre{ and more precisely Remark~\ref{rem:err_truncation}}).

 More precisely, we study the Markov process $(Z_t)_{t\geqslant0}=(X_t,Y_t)_{t\geqslant0}$ on $\R^{2d}$ that solves
\begin{equation}\label{eq}
  \left\{
      \begin{aligned}
        &\dd X_t = Y_t\dd t\\    
        &\dd Y_t = -\nabla U(X_t)\dd t - \gamma_t Y_t\dd t + \sqrt{2\gamma_t\beta_t^{-1}}\dd B_t\,,\\
      \end{aligned}
    \right.	 
\end{equation}
where $\gamma:\R_+\rightarrow \R_+$ is a friction parameter. We retrieve the settings of \cite{Pierro} with $\gamma_t=\beta_t$. We remark that the extension to Generalized Langevin processes as in \cite{GenLan} would not raise any particular difficulty, but we don't consider it for the sake of clarity.
We will work under the following set of assumptions :

\begin{assumption}\label{hyp}
\hspace{2em}
\begin{itemize}
\item $U:\R^d\mapsto\R$ is a  $\mathcal{C}^{\infty}$ potential such that $\min U=0$, $\lim_{|x|\rightarrow\infty}U(x)=\infty$, and there exists $\alpha_0>0$ such that : \[\int_{\R^d} e^{-\alpha_0 U(x)}\dd x < \infty.\]
\item The cooling schedule $\beta:\R_+\rightarrow\R_+$ is given by
\begin{equation}\label{bet}
\beta_t = \frac{\ln(e^{c\beta_0}+t)}{c},
\end{equation}
for some parameters $c>0$ and $\beta_0>0$.
\item The friction $\gamma:\R_+\rightarrow\R_+$ is a $\mathcal{C}^1$ function and there exists $\kappa>0$ such that for all $t\geqslant0$, $\gamma_t\geqslant\kappa$ and $\gamma'_t\leqslant\frac{1}{\kappa(1+ t)}$. 
In particular there exists $L$ such that $\gamma_t \leqslant L\beta_t$.
\item The critical height $ c^* $ of $U$ is finite, where $c^{\ast} = \sup_{x_1,x_2}c(x_1,x_2)$ and  
\[c(x_1,x_2) = \inf\left\{\max_{0\leqslant t \leqslant1}U(\xi(t))-U(x_1)-U(x_2) \right\}\]
where the infimum runs over $\left\{ \xi\in\mathcal{C}\left(\left[0,1\right],\R^d\right),\xi(0)=x_1, \xi(1)=x_2 \right\}$.
\end{itemize}
\end{assumption}


The condition $\min U =0$ is imposed for simplicity, it can always be enforced by changing $U$ to $U-\min U$. The specific form of $\beta$ is also made for simplicity, since it is known that, in order to study the convergence in probability for large time for simulated annealing algorithm, only logarithmic schedules are relevant. In particular, notice that, under Assumption~\ref{hyp}, the time-shifted process $(Z_{t_0+t})_{t\geqslant 0}$ for any $t_0\geqslant 0$ satisfies Assumption~\ref{hyp} with the same $U,c,\kappa,L$ and with $\beta_0$ replaced by $\beta_{t_0}$.


The critical height $c^*$ represents the largest energy barrier the process has to cross in order to go from any local minimum to any global one. In the classical overdamped case, disregarding the question of the behavior of $U$ at infinity, it is well known that, at least in the case where the global minimizer of $U$ is unique, the algorithm converges if $c>c^*$ (slow cooling) and has a positive probability to fail (i.e. to never visit a global minimum) if $c<c^*$ (fast cooling). We retrieve this dichotomy in the kinetic case. 

In the slow cooling case, we extend the results of \cite{Pierro}:

\begin{theorem}\label{Th}
Under Assumption~\ref{hyp}, assume furthermore that $c>c^*$. Then any solution $(Z_t)_{t\geqslant 0}$ of \eqref{eq} satisfies
\[\forall\ \delta>0,\quad \mathbb{P}(H(Z_t)\geqslant\delta)\rightarrow0.\]
\end{theorem}
Since $H(x,v)=U(x)+|v|^2/2$, this implies the convergence in probability of $X_t$ to the set of global minimizers of $U$.

On the other hand, if $c<c^*$, then the process might remain stuck in a region that contains no global minimum of $U$. In fact, a slightly stronger condition is required. Indeed, it is possible that $c^*>0$ with all minima of $U$ being global, in which case $H(Z_t)$ may go to zero with fast cooling schedule, while the law of $Z_t$ is not close to its local equilibrium. To be more precise, we need some additional definitions. For $x,y\in\R^d$, let 
\begin{equation}\label{barriersize}
    \tilde{c}(x,y)= \inf\left\{\max_{0\leqslant t \leqslant1}U(\xi(t))-U(x) \right\}
\end{equation}
where the infimum runs over $\left\{ \xi\in\mathcal{C}\left(\left[0,1\right],\R^d\right),\xi(0)=x, \xi(1)=y \right\}$. We define the depth of  $x\in \R^d$ by 
\begin{equation}\label{depth}
   \mathrm D(x)= \inf\left\{ \tilde c(x,y),\, y\in\R^d,\ U(y)< U(x)\right\},
\end{equation}
with $\inf\emptyset=\infty$. It is clear that $\mathrm{D}(x)=0$ if $x$ is not a local minimum of $U$. For $x\in\R^d$ with $\mathrm{D}(x)>0$ and $a\in (0,\mathrm{D}(x))$, we define the cup of bottom $x$ and height $a$ (this is the vocabulary of \cite{Hajek}) as 
\begin{equation}\label{cusp}
\mathrm{C}(x,a) = \{y\in\R^d,\ \tilde c(x,y) < a\}.
\end{equation}
We want to discard pathological cases where there are two non-global local minima $x_1,x_2$ with $U(x_1)=U(x_2)$, $\mathrm{D}(x_1)=\mathrm{D}(x_2) > c$ and $c(x_1,x_2)=c$, see Figure~\ref{Figure2} \pierre{(these cases do not prevent the result to hold, but the proof doesn't work directly, see Remark~\ref{rem:noncv} below. Notice that the problem is not that there are two local minima with the same energy level and depth, which is a pretty common situation as soon as there are some symmetries in the system; the problem is that the elevation $c(x_1,x_2)$ between them is \emph{exactly} the parameter $c$ chosen by the user in the cooling schedule which, now, is a very unlikely situation).} Hence, we work under the following condition.
\begin{assumption}\label{hyp2}
Assumption~\ref{hyp} holds and there exist $\tilde x\in \R^d$ with $\mathrm D(\tilde x) > c$ and $a\in (c,\mathrm{D}(\tilde x))$ such that for all $y\in \mathrm C(\tilde x,a)$, $\tilde c(y,\tilde x) < c$.
\end{assumption}

\begin{figure}
    \centering
    \includegraphics[scale=0.35]{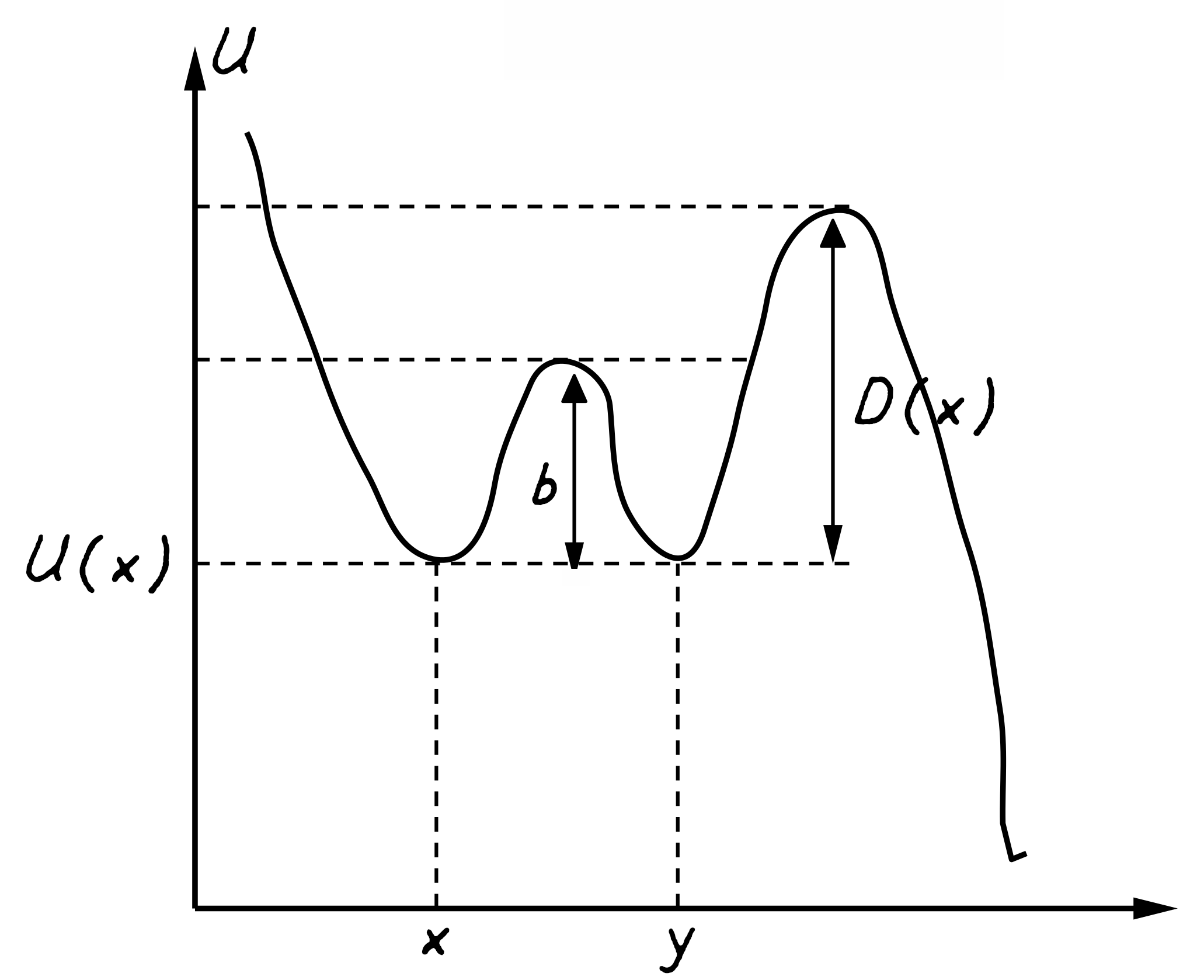}
    \caption{Case with two local non-global minima $x$ and $y$ at the same energy level. If $c<b$, then Assumption~\ref{hyp2} holds by chosing either $\tilde x=x$ or $\tilde x=y$ and any $a\in(c,b)$. If $c\in(b,D(x))$, again Assumption~\ref{hyp2} holds with $\tilde x\in\{x,y\}$ and any $a\in(b,D(x))$. However, if $c=b$, Assumption~\ref{hyp2} does not hold with $\tilde x\in\{x,y\}$.}
    \label{Figure2}
\end{figure}

\begin{theorem}\label{non}
Under Assumption~\ref{hyp2}, for all initial condition such  that $\mathbb P(X_0 \in \mathrm{C}(\tilde x,c))>0$ and all $\delta>0$, the solution of  \eqref{eq}  is such that 
\[  \mathbb{P}\left( X_t \in \mathrm{C}(\tilde x,c+\delta)\ \forall t\geqslant 0 \right) >0.\]
\end{theorem}
If $U$ has a finite number of minima and a unique global minimum, it is easily seen that there  exists a non-global minimum $x$ such that $\mathrm D(x)=c^*$. More generally, since for all $a<\mathrm D(x)$ the minimum value of $U$ over the cup $\mathrm C(\tilde x,a)$ is $U(\tilde x)$, if there exists  a non-global minimum $x$ of $U$ with depth $\mathrm D(x)>c$,  the previous results implies that, with positive probability, $\inf_{t\geqslant0}U(X_t) \geqslant U(\tilde x)>\min U$. Moreover, even if the probability to start in $\mathrm C(\tilde x,c)$ is initially zero, due to the controllability of the process, it is positive for all positive times (see e.g. \cite[proposition 5]{Pierro}). As a conclusion, we immediately get the following:

\begin{corollary}\label{cor:non-conv}
Under Assumption~\ref{hyp2}, for all initial condition and all $t_0>0$,
\[\mathbb P\po U(X_t) \geqslant U(\tilde x)\ \forall t\geqslant t_0\pf>0.\]
\end{corollary}

Notice that, in practice, one can keep track of $X_{s(t)}$ where \[s(t) = \inf\{w\in[0,t], U(X_w) = \min_{w'\in[0,t]}U(X_{w'})\},\]so that $X_{s(t)}$ may converge to a minimizer of $U$ even if $X_t$ does not. However, our results show that this doesn't solve the issue of non-convergence for fast cooling schedules.

\pierre{
\begin{remark}\label{rem:noncv}
In Figure~\ref{Figure2}, in the case $c=b$ (so that Assumption~\ref{hyp2} do not hold) we cannot deduce from our results that the process stays stuck with positive probability in the cusp  $\mathrm C(x,b)$ because, for any $\delta>0$, $\mathrm C(x,b+\delta)$ contains $y$ (contrary to  $\mathrm C(x,b)$). In fact it is clear that the process can stay stuck with positive probability in  $\mathrm C(x,b+\delta)$ for any $\delta>0$ (so that the conclusion of Corollary~\ref{cor:non-conv} also holds), but we cannot deduce it from our proof which requires that the critical depth within the cusp (i.e. for a suitable modification of the potential which only consider the local situation of this cusp, see Section~\ref{fullprocess} and Figure~\ref{fig:my_label}) is strictly smaller than $c$ (in order to apply a variation of Theorem~\ref{Th}), while it is exactly  $b=c$ in this example. We refer to \cite{Miclo2} where a fine analysis is conducted on a related question on finite graphs. 
\end{remark}
}

\bigskip

Finally we address the case of faster than logarithmic cooling schedules. For simplicity we restrict this study to the case of a constant friction parameter $\gamma$ (although as discussed at the end of  the proof it can be extended to the non-constant case with suitable conditions on $\gamma$ depending on $\beta$). In the following we do not assume that $t\mapsto \beta_t$ is increasing, and possibly $\beta_t=+\infty$ for some $t$.

\begin{theorem}\label{thmFast}
Assume that $\gamma_t=\gamma$ is constant, that $t\in\R_+\mapsto \beta_t\in(0,+\infty]$ is piecewise continuous with $\ln(t)=o(\beta_t)$ as $t\rightarrow +\infty$, that $U\in\mathcal C^\infty(\R^d)$ and that $x_*\in\R^{d}$ is a non-degenerate local minimum of $U$, i.e. $\na U(x_*)=0$ and $\na^2 U(x_*)>0$. Then there exist $C,r>0$ such that, denoting
\[\varepsilon_t =  C \po e^{-rt} + \sup_{s\geqslant t/2} \ln(s) \beta_s^{-1}\pf\,,\]
the following holds. For all $\delta>0$ and all initial condition $z_0=(x_0,y_0)\in\R^{2d}$ with $|x_0-x_*|\leqslant \delta/2$, the solution $Z=(X,Y)$ of \eqref{eq} is such that
\[\mathbb P_{z_0} \po |X_t - x_*| \leqslant \min(\delta,\sqrt{\varepsilon_t}) \ \forall t\geqslant 0\pf >0\,.\]
\end{theorem}





\medskip

The rest of the paper is dedicated to the proofs of Theorems~\ref{Th}, \ref{non} and \ref{thmFast}. It is organised as follows. The main step of the proofs in the logarithmic case are exposed in Section~\ref{main}, while technical intermediary results are postponed to Sections~\ref{sec:auxiliary}, \ref{hypo1} and \ref{hypo2}. More precisely, Section~\ref{EnergyBounds} is dedicated to the proof of a uniform in time energy bound, which is the main ingredient in the proof that  the process goes back infinitely many times to a compact set. A result of small-time conditional regularization is proven in Section~\ref{sec:regular}, which is used to replace deterministic initial conditions by smooth distributions (with some quantitative bounds).  Section~\ref{hypo1} presents hypocoercivity estimates  similar to those of \cite{Pierro}, which are used to prove the convergence of the algorithm in the slow cooling case. In the fast logarithmic cooling case, similar estimates have to be established in higher order Sobolev norms, which is the topic of Section~\ref{hypo2}. Section~\ref{SecFast} is dedicated to the faster than logarithmic case, with the proof of Theorem~\ref{thmFast}.


\section{Main steps of the proofs}\label{main}

The sketch of the proofs  is the following.

 The  first point is to get uniform in time moment estimates. This would classically be done using Lyapunov arguments, but this would typically require some assumption on $\na U$, which we want to avoid. We adapt an argument from \cite{FourTar}. From these estimates, we get that $\liminf_{t\rightarrow +\infty} H(Z_t)$ is almost surely finite, i.e. there exists a (random) compact set $\{H \leqslant A\}$ which will be visited infinitely often by the process. 
 
 The second step is to prove that, for any $A>0$, there exists $A'>A$ such that, provided the process is in $\{H\leqslant A\}$ at some time $t_0$, there is a probability at least, say, $1/4$, that the process remains in $\{H\leqslant A'\}$ for all times $t\geqslant t_0$. This is reminiscent of the study in \cite{HoStKu} of fast cooling schedules,  where it is proven that there is a positive probability that, starting in a potential well of depth larger than $c$, the process never climbs high enough to exit the well (as in Theorem~\ref{non}). We will follow a similar proof, except that we have to check the dependency of the estimates with respect to $t_0$ or, equivalently by taking $t_0=0$, to $\beta_0$. That way, we will conclude that, each time the process goes below $A$, it has a probability $1/4$ to never go above $A'$ again so that, if it goes below $A$ infinitely often, eventually it will stay below $A'$. 

Combining the two previous steps, we get that the process is almost surely bounded. It is thus sufficient to prove the convergence of the process when the position space is a compact torus, which is then similar to \cite{Pierro}, but without the issue of the behavior of $U$ at infinity.

The strategy for the non-convergence in the fast cooling case is similar to \cite{HoStKu}, the technical difficulties coming from the degeneracy of the process. Indeed the estimates used in \cite{Pierro} in the kinetic case to prove convergence are too weak to get that the process has a positive probability to stay forever below some energy level. On the other hand the estimates of \cite{HoStKu} are based on hypercontractivity of elliptic diffusions on compact manifold. We are not aware of similar results for hypoelliptic diffusions, and thus we overcome this difficulty by working with higher Sobolev norms. 

The faster than logarithmic case is similar and somehow simpler: in that case the process has a positive probability to converge to $(x_*,0)$, it is thus sufficient to linearize $\nabla U$ at this point and to study the corresponding Gaussian process.

\subsection{Return to a compact set}\label{return}

Diffusions defined by \eqref{eq} are time-inhomogeneous Markov processes with generator :
\begin{equation}\label{gen}
L_{t} = y\cdot \nabla_x - (\gamma_t y + \nabla_xU)\cdot \nabla_y + \gamma_t \beta_t^{-1}\Delta_y. 
\end{equation}
First, let us check that  the process is well-defined.

\begin{proposition}\label{defined}
Under Assumption~\ref{hyp}, let $z_0 \in\R^{2d}$. There exists a unique process $Z=(X,Y)$ that solves ~\eqref{eq} with $Z_0=z_0$. It is non-explosive (i.e. defined for all $t\geqslant 0$) and, for all $t\geqslant 0$, $\E(H(Z_t))\leqslant H(z_0) + dLt$.
\end{proposition}

\begin{proof}
Since the coefficients of the diffusion \eqref{eq} are all smooth, there is existence and uniqueness until a time  $\xi$ of explosion. For all $x,y\in\R^d$ and $t\geqslant 0$,
\[L_{t}H(x,y) 
= \gamma_t\beta_t^{-1}d-\gamma_t|y|^2 \leqslant Ld.\]
Considering for $N\in\N$ the stopping time $\tau_N = \inf\left\{t\geqslant0;H(Z_t)\geqslant N\right\}$, we get
\[\E(H(Z_{\tau_N\wedge t})) \leqslant H(z_0) + dLt\]
which implies
\[\mathbb{P}(\tau_N <t) = \mathbb{P}(H(Z_{\tau_N\wedge t})\geqslant N) \leqslant \frac{H(z_0) + Ldt}{N} \underset{N\rightarrow\infty}\longrightarrow 0\,,\]
hence $\xi<+\infty$ almost surely, which concludes. 
\end{proof}


The first main point  is to  to strengthen the energy bound of Proposition~\ref{defined} to a uniform in time estimate.  For $f$ a probability law or density on $\R^{2d}$, we write $\mathbb E_{f}$ and $\mathbb P_f$ expectations and probabilities with respect to the process $Z$ solving \eqref{eq} with initial condition $Z_0$ distributed according to $f$. If $f=\delta_z$ for some $z\in\R^{2d}$, we write $\E_z$ and $\mathbb{P}_z$ instead.


\begin{lemma}\label{energybounds}
Under Assumption~\ref{hyp}, there exists $b> \alpha_0 $ that depends only on $U$  such that the following holds. Provided $\beta_0 \geqslant b$ then, for any $\mathcal C^\infty$ probability density $f_0$ with compact support, 
\[\sup_{t\geqslant0} \E_{f_0}\left(H(Z_t)\right)\leqslant \frac{\kappa^{\beta_0}(f_0)+\ln(\mathcal{Z}_{\alpha_0})}{\beta_0-\alpha_0}\,,\]
where \[\kappa^{\beta_0}(f_0)=\int_{\R^{2d}} f_0\ln \left(1+f_0e^{\beta_0H}\right)\,,\qquad \mathcal{Z}_{\alpha_0}=\int_{\R^{2d}} e^{-\alpha_0H} .\]
\end{lemma}

The proof is postponed to Section~\ref{EnergyBounds}.  The next result enables the use of Lemma~\ref{energybounds} when the initial condition is not smooth. 

\begin{lemma}\label{changcondinit}
Under Assumption~\ref{hyp} with $c>c^*$, fix $A>1$ and $\varepsilon>0$. Then there exist $t^*,b_A,C^1_A,C>0$ that do not depend on $\beta_0$ such that, for all $\beta_0 \geqslant b_A$ and $z_0\in \R^{2d}$ with $H(z_0) \leqslant A$, the following holds. Writing $\B=\left\{\sup_{t\leqslant t^*}H(Z_t)\leqslant A+1 \right\}$,
\[\mathbb{P}_{z_0}(\B)\geqslant 1-\varepsilon.\]
Moreover, the law at time $t^*$ of the process \eqref{eq}  with initial condition $Z_0=z_0$ conditioned on the event $\B$ has a density $f^c_{t^*}$ that satisfies
\[f^c_{t^*}\leqslant Ce^{C^1_{A}\beta_0}\mathbbm{1}_{\left\{H \leqslant A + 1\right\}}.\]
\end{lemma}

This is proven in Section~\ref{sec:regular}. The two previous lemmas yield the following.

\begin{proposition}\label{retourcompact}
Under Assumption~\ref{hyp}, $\liminf_{t\rightarrow\infty} H(Z_t)<\infty $ almost surely.
\end{proposition}

\begin{proof}
For $t$ large enough, $\beta_{t}\geqslant b$ where $b>\alpha_0$ is the constant from Lemma~\ref{energybounds}. Notice that, for $t_0\geqslant 0$, $(Z_{t_0+t})_{t\geqslant0}$ solves \eqref{eq} except that $\beta_0$ is replaced by $\beta_{t_0}$ and $\gamma$ is also time-shifted. Hence, by the Markov property, without loss of generality, we can assume that  $\beta_0\geqslant b$. Moreover, by conditioning with respect to the initial condition, it is sufficient to consider the case of a deterministic initial condition $z_0\in\R^{2d}$.

Fix $\varepsilon>0$ and $A=H(z_0)$. From Lemma~\ref{changcondinit}, there exist $t^*,C'>0$  such that $\mathbb P_{z_0}(\mathcal{B})\geqslant1-\varepsilon$ where $\B=\left\{\sup_{t\leqslant t^*}H(Z_t)\leqslant A+1 \right\}$, and such that
 the law of the process at time $t^*$ and conditioned on $\mathcal{B}$  has a density $\tilde{f}_{t^*}\leqslant C'\mathbbm{1}_{\{H\leqslant A+1\}}$. Let  $C>0$  and $f_0$ be  a $\mathcal{C}^{\infty}$ probability density on $\R^{2d}$ with compact support  such that $Cf_0 \geqslant C'\mathbbm{1}_{\{H\leqslant A+1\}}$. Then, denoting by $(\mathcal F_t)_{t\geqslant 0}$ the filtration associated with $(Z_t)_{t\geqslant 0}$, by the strong Markov property and Lemma~\ref{energybounds}, for all $t\geqslant 0$, 
\[\E_{z_0}(H(Z_{t^*+t})|\B) = \E_{z_0}(\E(H(Z_{t^*+t})|\mathcal{F}_{t^*})|\B) \leqslant C\E_{f_0}(H(Z_{t}))\leqslant C\frac{\kappa^{\beta_{0}}(f_0)+\ln(\mathcal{Z}_{\alpha_0})}{\beta_{0}-\alpha_0}.\]
As a consequence, by Fatou's Lemma,
 \[ \E_{z_0}\left(\liminf_{t\rightarrow\infty} H(Z_t)|\B\right)\leqslant \sup_{t\geqslant 0} \E_{z_0}(H(Z_t)|\B) < \infty. \]
Finally
 \[\mathbb{P}_{z_0}\left(\liminf_{t\rightarrow\infty} H(Z_t)=\infty\right) \leqslant \mathbb{P}_{z_0}(\B^c) + \mathbb{P}_{z_0}\left(\liminf_{t\rightarrow\infty} H(Z_t)=\infty|\B\right) \leqslant \varepsilon\,,\] 
 which concludes since $\varepsilon$ is arbitrary.
\end{proof}

\subsection{Position in a compact set}\label{posicompset}

As in \cite{FourTar}, with Proposition~\ref{retourcompact} at hand, we can now focus on the behavior of the process when the position is in a compact set, ignoring the behavior of $U$ at infinity.  To this end, we fix some parameter $K>0$, let $L_K>0$  be such that $\left\{U\leqslant K\right\}\subset\left[-(L_K-1),(L_K-1)\right]^d$ and consider the torus $M_K=(\R/2L_K\Z)^d$ \pierre{(i.e. we consider periodic boundary conditions, which will be technically simpler than e.g.  reflecting  boundary conditions)}.  We now define a process $(X^K_t,Y^K_t)$ on $M_K\times\R^d$ and a potential $U^K:M_K\mapsto\R$ to replace the initial one.

More precisely, write $\theta_K:\R^d\mapsto M_K$ the canonical projection and $\tilde{M}_K=\left]-L_K,L_K\right]^d$, so that  $M_K=\theta_K(\tilde{M}_K)$. For a non-negative function $V:\R^d\rightarrow\R$,  we define the critical height $c^*(V)$ with the same definition as $c^*$ except that $U$ is replaced by $V$.

 Let $\tilde{U}^K\in\C^{\infty}(\R^d)$ be equal to $U$ on \lj{some open set $O_K$ containing} $\left\{U\leqslant K\right\}$, non-negative, $2L_K$-periodic, and such that $c^*_K = c^*(\tilde{U}^K) \leqslant c^*$. Such a function exists from \cite[Notation 9]{FourTar}. We write  $U^K$ the corresponding function on $M_K$, given by  $U^K(\theta_K(x))=\tilde{U}^K(x)$, and $H_K(x,y)=U^K(x)+|y|^2/2$. 

Finally, given the same Brownian motion as in \eqref{eq}, we consider $Z^K=(X^K,Y^K)$ the process on $M_K\times\R^d$ that solves 
\begin{align}
  \left\{
      \begin{aligned}\label{eq1}
        &\dd X^K_t = Y^K_t\dd t\\    
        &\dd Y^K_t = -\nabla_xU^K(X_t)\dd t - \gamma_t Y^K_t\dd t + \sqrt{2\gamma_t\beta_t^{-1}}\dd B_t\,,
      \end{aligned}
    \right. 
\end{align}
with $(X_0^K,Y_0^K)=(\theta_K(X_0),Y_0)$. Write \lj{$\tau_K=\inf\left\{t\geqslant0,\ X_t\notin O_K\right\}$}. Then, by design,
\begin{equation}
\label{eq:egalte_proc}
(U^K(X^K_t),Y_t^K)_{t\leqslant \tau_K} =(U(X_t),Y_t)_{t\leqslant \tau_K}.
\end{equation}
For $\beta>0$,  write $\mu^K_{\beta}$ the probability \pierre{measure} on $M_K\times\R^d$ with density proportional to $e^{-\beta H_K} $. 

\begin{lemma}\label{conveq}
Fix $\delta,\alpha>0$. There exists $C>0$ such that for all $\beta> 0$, 
\[ \mu_{\beta}^K\left(H_K >\delta\right)\leqslant Ce^{-\beta(\delta-\alpha)}. \]
\end{lemma}
\lj{\begin{proof}
For completeness, we recall the short proof of \cite[section 3.1]{Pierro}. First,
\[\int_{M_K\times \R^d} e^{-\beta H_K(z)}\dd z \geqslant \int_{\left\{H_K\leqslant \alpha \right\}} e^{-\beta H_K(z)}\dd z \geqslant e^{-\beta\alpha}/C,\]
for some constant $C>0$ because $\left\{H^K\leqslant \alpha \right\}$ is compact. Likewise, $\left\{\delta \leqslant H^K\leqslant \delta + |y| \right\}$ is compact and
\begin{align*}
\int_{\left\{H^K\geqslant \delta \right\}} e^{-\beta H_K(z)}\dd z &\leqslant \int_{\left\{\delta \leqslant H_K\leqslant \delta + |y| \right\}} e^{-\beta H_K(z)}\dd z + \int_{\left\{H_K\geqslant \delta + |y| \right\}} e^{-\beta H_K(z)}\dd z \\ &\leqslant e^{-\beta\delta } \po C + C\int_{\R^d} e^{-|y|} \dd y\pf
\end{align*}
for some $C>0$. The ratio of those two inequalities concludes the proof.
\end{proof}}
The goal of this section is to prove a similar result but with the law of $Z_t^K$ instead of $\mu_\beta^K$, with an explicit dependency of the constants in term of $\beta_0$.

Denote by $f^K_t$ the law of $Z_t^K$. Since $\na U^K$ and its derivatives are all bounded, we will be able to show that $f_t^K$ is smooth, see Section~\ref{hypo1}. Consider the relative density $ h^K_t=f^K_t/\mu^K_{t}$. We start with a  uniform in time quantitative bound of the norm of $h_t^K$  in $L^2\left(\mu^{K}_{t}\right)$, for nice initial conditions.


\begin{proposition}\label{evol0}
Under Assumption~\ref{hyp} with $c>c^*$, let $ K\geqslant 1$. There exists $b_K>0$, which depends on $U$ and $K$, such that, if $\beta_0\geqslant b_K$ and $f_0\in\C^{\infty}(M_K\times\R^d)$ with compact support, then for all $t\geqslant0$,
\begin{multline*}
\int_{M_K\times\R^d} \po h^K_t\pf ^2\dd \mu_{t}^K    \\ \leqslant 1+ \int_{M_K\times\R^d}\left(\left|(\na_x+\na_y)h_0^K\right|^2+4\sqrt{\gamma_0^{-1}\beta_0}(\|\na^2_x U^K\|_{\infty}+1+\gamma_0)^2\po h_0^K\pf ^2\right)\dd \mu_{0}^K  .
\end{multline*}
\end{proposition}

The proof of this proposition is postponed to Section~\ref{hypo1}. 


\begin{lemma}\label{integ}
Under Assumption~\ref{hyp} with  $c>c^*$, fix  $A>1$, $\varepsilon=1/4$ and consider $\mathcal B,C_A^1,b_A,t^*$ as in  Lemma~\ref{changcondinit}. Set $D_A=2C_{A}^1 + A + 4 + 4c$ and $K_{A}=D_{A}+1$. 
There exist $C_{A}>0$ and $b_A'\geqslant b_A$ that do not depend on $\beta_0$ such that, for all $\beta_0 \geqslant b_A'$ and $z_0\in \R^{2d}$ with $H(z_0) \leqslant A$, we have that, $\mathbb P_{z_0 }$-a.s., $\sup_{t\in\left[0,t^*\right]}H_{K_A}\left(Z^{K_A}_t\right)\mathbbm{1}_{\B}< D_{A}$,  and for all $t\geqslant 0$ 
\[\mathbb P_{z_0}\po \left. H_{K_A}\left(Z^{K_A}_{t+t^*}\right)\geqslant D_{A}\right|\B\pf\leqslant \frac{C_{A}}{(e^{c\beta_0}+t)^2}.\]
\end{lemma}

\begin{proof}
The first point is clear since, by definition of $\B$, $\sup_{\left[0,t^*\right]}H_{K_A}(Z^{K_A}_t)\mathbbm{1}_{\B}< A+1 \leqslant D_A$.

In the rest of the proof we denote by $C$ different constants.

Let us introduce a function which will serve as a new initial condition. Let $v_d$ and $w_d$ be respectively the volumes of $\left\{H\leqslant A+1 \right\}$ and  $\left\{H\leqslant A+2 \right\}$ in $M_K\times\R^d$, and $f_0\in\C^{\infty}$ be such that $(2v_d)^{-1}\mathbbm{1}_{\left\{H\leqslant A+1\right\}}\leqslant f_0 \leqslant 2(w_d)^{-1}\mathbbm{1}_{\left\{H \leqslant A+2\right\}}$. Notice that $f_0$ does not depend on $\beta_0$, and that $\na f_0$ is bounded since $f_0$ is smooth on a compact set. 
As in the proof of Lemma~\ref{retourcompact},  from Lemma~\ref{changcondinit} and the strong Markov property,
\begin{align*}
\procz{z_0}{H_{K_A}\left(Z^{K_A}_{t+t^*}\right)\geqslant D_{A}}{\B} 
 &\leqslant  2v_dCe^{C^1_A\beta_0}\mathbb{P}_{f_0}\left(H_{K_A}\left(Z^{K_A,\beta_{t^*}}_{t}\right)\geqslant D_{A}\right)\,,
\end{align*}
where $Z^{K_A,\beta_{t^*}}$ is a process similar to $Z^{K_A}$ except that $\beta_0$ has been replaced by $\beta_{t^*}$ and $(\gamma_t)_{t\geqslant 0}$ by $(\gamma_{t+t^*})_{t\geqslant 0}$. Denote by $h_t^*$ the relative density of the law of $Z^{K_A,\beta_{t^*}}_t$ with respect to $\mu_{\beta_{t+t^*}}^{K_A}$. By the Cauchy-Schwartz inequality,  
\begin{align}
\mathbb{P}_{f_0}\left(H_{K_A}\left(Z^{K_A,\beta_{t^*}}_{t}\right)\geqslant D_{A}\right) &= \int_{\left\{H_{K_A}\geqslant D_{A}\right\}} h_t^* \dd \mu_{\beta_{t+t^*}}^{K_A}\nonumber \\ &\leqslant \sqrt{\int_{\R^d\times\R^d} (h_t^*)^2\dd \mu_{\beta_{t+t^*}}^{K_A}}\sqrt{\mu_{\beta_{t+t^*}}^{K_A}(H_{K_A}\geqslant D_{A})}.\label{eq:CS}
\end{align}
For the second term, applying Lemma~\ref{conveq}  with $\alpha=1$, 
\[\mu_{\beta_{t+t^*}}^{K_A}(H_{K_A}\geqslant D_{A})  \leqslant  Ce^{-\beta_{t+t^*}(D_A-1)}\,.\]
For the first term, Proposition~\ref{evol0} applies if $\beta_0\geqslant b_{K_A}$, which yields
\begin{align*}
\int_{M_K\times\R^d} &\left(h_t^*\right)^2\dd \mu_{\beta_{t+t^*}}^{K_A} \\ &\leqslant 1+  \int_{M_K\times\R^d}\left(\left|(\na_x+\na_y)h_0^*\right|^2+4\sqrt{\gamma_{t^*}^{-1}\beta_{t^*}}(\|\na^2_x U^K\|_{\infty}+1+\gamma_{t^*})^2\left(h_0^*\right)^2\right)\dd \mu_{\beta_{t^*}}^{K_A}\\
&\leqslant  C\beta_{t^*}^{\frac{5}{2}}\int_{\left\{H_{K_A}\leqslant A+2\right\}} \frac{1}{\mu_{\beta_{t^*}^{K_A}}(z)} \dd z 
\end{align*}
for some $C>0$ that does not depend on $\beta_0$, where we used Assumption~\ref{hyp} and a uniform bound on $f_0$ and $\na f_0$.  For any $\beta\geqslant 1$, using that $e^{-\beta U^{K_A}(x)}\leqslant 1$,
\[ \int_{M_{K_A}} e^{-\beta U^{K_A}(x)} \dd x\leqslant (2L_{K_A})^d\qquad \text{and} \qquad \int_{M_{K_A}\times\R^d} e^{-\beta H_{K_A}(z)} \dd z \leqslant (2\sqrt{2\pi}L_{K_A})^d .\]
Hence,   
\[\int_{M_{K_A}\times\R^d}  \left(h_t^*\right)^2\dd \mu_{\beta_{t+t^*}}^{K_A} \leqslant  C\beta_{t^*}^{\frac{5}{2}}\int_{\left\{H_{K_A}\leqslant A+2\right\}} e^{\beta_{t^*}H_{K_A}(z)} \dd z \leqslant Ce^{\beta_{t^*}(A+3)}  \]
Everything put together    gives, if $\beta_0 \geqslant \max(1,b_{K_A},b_A)$, using the monotonicity of $t\mapsto\beta_t$, we conclude with
\[\proc{H_{K_A}\left(Z^{K_A}_{t+t^*}\right)\geqslant D_{A}}{\B}  \leqslant Ce^{-\frac{1}{2} \beta_{t^*+t}\left(D_{A}-1-2C_A^1-A-3\right)}  \leqslant Ce^{-2c\beta_t}  = \frac{C}{(e^{c\beta_0}+t)^2} .\]
\end{proof}

In fact, a similar proof already yields the convergence of the kinetic annealing on the compact torus:

\begin{proposition}
\label{prop:CVtore}
Under Assumption~\ref{hyp} with $c>c^*$,   for all $K>1$ and $\delta>0$,
\[\mathbb{P}(H_K(Z^K_t)>\delta)\underset{t\rightarrow+\infty}\longrightarrow 0.\]
\end{proposition}
\begin{proof}
As in the proof of Proposition~\ref{retourcompact}, by the Markov property, without loss of generality we assume that $\beta_0\geqslant b_K$ and that the initial condition is a Dirac mass at some $z_0\in  M_K\times \R^d$. Fix $\varepsilon>0$, let $A=H(z_0)$ and, from Lemma~\ref{changcondinit} (applied to the process $Z^K$ rather than $Z$, the proof is the same) let $t_0,C>0$ be such that the event $\mathcal B = \{\forall t\in[0,t_0], H(Z^K)\leqslant A+1\}$ has a probability larger than $1-\varepsilon$ and   the law of $Z_{t_0}^K$ conditioned on $\mathcal{B}$ has a density $\tilde{f}^K_{t_0}\leqslant C'\mathbbm{1}_{\left\{H_K\leqslant A+1\right\}}$. Let $C>0$ and $f_0$ be a smooth probability density on $M_K\times\R^d$ with compact support such that $C'\mathbbm{1}_{\left\{H_K\leqslant A+1\right\}}\leqslant Cf_0$. We then have:
\[
\proc{H_K(Z^K_{t+t_0})>\delta}{\B} =\ec{\proc{H_K(Z^K_{t+t_0})>\delta}{\F_{t_0}}}{\B} \leqslant C\mathbb{P}_{f_0}\left(H_K(\tilde Z_t)>\delta\right)\,,
\]
where $\tilde Z$ is a process similar to $Z^{K_A}$ except that $\beta_0$ has been replaced by $\beta_{t_0}$ and $\gamma$ has been time-shifted. Denoting by $\tilde{f}_t$ the law of $\tilde Z_t$ (with initial condition $f_0$),   Proposition~\ref{evol0} means that $t\mapsto \|\tilde  f_t/\mu^K_{\beta_{t+t_0}}\|_{L^2(\mu^K_{\beta_{t+t_0}})}$ is bounded.
 As in the previous proof, applying the Cauchy-Schwarz inequality and Lemma~\ref{conveq} with $\alpha=\delta/2$ then yields :
\[\mathbb{P}_{f_0}^2\left(H_K(\tilde Z_t)>\delta\right)\leqslant C \mu^K_{\beta_{t+t_0}}(H_K>\delta) \leqslant Ce^{-\frac{\beta_t\delta}{2}}\rightarrow 0\,.\]
As a consequence,
\[\limsup_{t\rightarrow\infty}\mathbb{P}(H_K(Z^K_t)>\delta)\leqslant \mathbb{P}\left(\mathcal B^c\right) + \limsup_{t\rightarrow\infty}\proc{H_K(Z^K_t)>\delta}{\B }\leqslant\varepsilon\,,\]
which concludes since $\varepsilon$ is arbitrary.
\end{proof}

\subsection{Localization and convergence}\label{Loca}

%
%
Building upon the results of the previous section, we can now prove the following.

\begin{proposition}\label{boundedness}
Under Assumption~\ref{hyp} with $c>c^*$, fix some $A>1$. There exist $b_{A}>1$, $K_A>A$ which depends on $A$, $U$, $c$, and $\kappa$ but not $\beta_0$ such that, for all $\beta_0\geqslant b_{A}$ and all initial condition $z_0\in\{H\leqslant A\}$,
\[\mathbb{P}_{z_0}\left(\sup_{t\geqslant0}H\left(Z_t\right)\leqslant K_A\right)\geqslant \frac{1}{4}.\]
\end{proposition}

\begin{proof}
It is enough to show the same result for the process $Z^{K_A}$ since, from \eqref{eq:egalte_proc},
\lj{
\begin{align*}
&\left\{\sup_{t\geqslant0}H(Z_t)\leqslant K_A\right\} \\ &= \left\{\sup_{t\geqslant0}H(Z_t)\leqslant K_A,(U^K(X^{K_A}_t),Y_t^{K_A})_{t\leqslant \tau_{K_A}} =(U(X_t),Y_t)_{t\leqslant \tau_{K_A}}, \tau_{K_A}=\infty \right\} \\ &= \left\{\sup_{t\geqslant0}H_{K_A}(Z^{K_A}_t)\leqslant K_A,(U^K(X^{K_A}_t),Y_t^{K_A})_{t\leqslant \tau_{K_A}} =(U(X_t),Y_t)_{t\leqslant \tau_{K_A}}, \tau_{K_A}=\infty \right\} \\ &=\left\{\sup_{t\geqslant0}H_{K_A}(Z^{K_A}_t)\leqslant K_A\right\}.
\end{align*}}
We use the definitions and notations of Lemma~\ref{integ}.
Let $\Psi_{A}\in\C^{\infty}(\R_+,[0,1])$ equal to $0$ outside of $\left[D_{A},K_{A}\right]$ and such that $\Psi_{A}\left(\frac{D_{A}+K_{A}}{2}\right)=1$, and let $\Phi_{A}=\Psi_A\circ H^{K_{A}}$. Then, 
 there exists $C>0$, independent of $\beta$, such that, for all $t\geqslant 0$, $|L_t\Phi_{A}|\leqslant C(1+\gamma_t)\mathbbm{1}_{\left\{H^{K_{A}}\geqslant D_{A}\right\}}$. 
 
Recall the definition $\B=\left\{\sup_{t\leqslant t^*}H(Z_t)\leqslant A+1 \right\}$. From Ito's formula,  
\[\Phi_A\left(Z_{t+t^*}^{K_A}\right)\mathbbm{1}_{\B}=\left(M_t + R_t\right)\mathbbm{1}_{\B}  \]
where $M_t$ is a local martingale and $R_t=\int_{t^*}^t L_s\Phi_A(Z^{K_A}_s)\dd s$.
We then get that for $\beta_0\geqslant b_A'$, 
\begin{align*}
\E\left(\sup_{t\geqslant0} |R_t|\mathbbm{1}_\B\right) &\leqslant C\int_0^{\infty}(1+\gamma_s) \ec{\mathbbm{1}_{\left\{H^{K_{A}}\left(Z_{s+t^*}^{K_A}\right)\geqslant D_{A}\right\}}}{\B} \dd s \\ &\leqslant C\int_0^{\infty}(1+L\beta_s) \proc{H\left(Z^{K_A}_{s+t^*}\right)\geqslant D_{A}}{\B} \dd s \\ &\leqslant C\int_0^{\infty} \frac{1+L\ln(e^{c\beta_0}+t)}{c(e^{c\beta_0}+t)^2} \dd s.
\end{align*}
We take  $\beta_0$ large enough to get that, by Markov's inequality,   $\mathbb{P}(\mathcal E|\B)\geqslant 3/4$, where $\mathcal E=\left\{\sup_{t\geqslant0} |R_t|\leqslant 1/10\right\}$.
 Now,
\[\mathcal E\cap\B\cap \left\{\sup_{t\geqslant0}H_{K_A}(Z^{K_A}_t)\geqslant K_A\right\} \ \ \subset \ \mathcal E\cap\B \cap \left\{\sup_{t\geqslant0}M_t\geqslant \frac{9}{10}\right\}\,,\]
so that
\[\proc{\sup_{t\geqslant 0}H_{K_A}\left(Z^{K_A}_t\right)\geqslant K_A }{\B} \leqslant \proc{\mathcal E^c}{\B} + \proc{\mathcal E, M_t \text{ up-crosses } \left[\frac{1}{10},\frac{9}{10}\right]}{\B} \,.\] 
Consider the stopping time 
 \[ \sigma = \inf\left\{t\geqslant 0, M_t\notin\left[-\frac{1}{10},\frac{11}{10}\right]\right\}.\]
Then $M_{t\wedge\sigma}$ is a bounded local martingale, hence a martingale. Moreover, since $\Phi_A$ takes values in $[0,1]$,   $M_t=M_{t\wedge\sigma}$ for all $t\geqslant 0$ on $\mathcal E\cap\B$. 
By Doob's up-crossing inequality, for any $T>0$, the probability  that $M_t$ up-crosses $[1/10,9/10]$ before time $T$ is bounded  by $\E(M_{T\wedge\sigma}-\frac{1}{10})/(8/10) \leqslant \frac{1}{4}$. As a consequence,
 \[\proc{\mathcal E, M_t  \text{ up-crosses } \left[\frac{1}{10},\frac{9}{10}\right]}{\B} \leqslant \frac{1}{4}\,. \]
 As a conclusion,
\begin{multline*}
\mathbb{P}\left(\sup_{t\geqslant 0}H_{K_A}\left(Z^{K_A}_t\right)\geqslant K_A \right) \\ \leqslant \mathbb{P}\left(\B^c\right) + \mathbb{P}(\mathcal E^c|\B) + \proc{\mathcal E, M_t  \text{ up-crosses } \left[\frac{1}{10},\frac{9}{10}\right]}{\B}  \leqslant \frac{3}{4}.
\end{multline*}
\end{proof}

As announced, combining this result with Proposition~\ref{retourcompact} yields the following. 

\begin{proposition}\label{prop:born{\'e}}
Under Assumption~\ref{hyp} with $c>c^*$, almost surely, $\sup_{t\geqslant 0} H(Z_t) <+\infty$.
\end{proposition}

\begin{proof}
For $A>0$, let 
\[\Omega_{A}=\left\{\liminf_{t\rightarrow\infty} H(Z_t)<A\right\}.\]
It is sufficient to show that, for all $A>0$,   $(Z_t)_{t\geqslant0}$ is bounded on $\Omega_{A}$ since, from Proposition~\ref{retourcompact}, $\mathbb{P}\left(\cup_{A\geqslant 0} \Omega_{A}\right) = 1$. Hence, we fix $A>0$.

Let $b_{A},K_A$ be as in Proposition \ref{boundedness} and $S_0=t_A$ where $t_A$ satisfies $\beta_{t_A}\geqslant b_A$. By induction, for $k\in\N$, define $T_k = \inf\left\{t\geqslant S_{k}; H(Z_t)\leqslant A\right\}$ and $S_{k+1}=\inf\left\{t\geqslant T_k; H(Z_t)\geqslant K_A \right\}$. From Proposition~\ref{boundedness} and the Markov property, 
\begin{align*}
\mathbb{P}(S_{k+1}<\infty|S_k<\infty) &= \mathbb{P}(S_{k+1}<\infty,T_k<\infty|S_k<\infty) \\ &= \E(\mathbbm{1}_{T_k<\infty}\mathbb{P}(S_{k+1}<\infty|\mathcal{F}_{T_k})|S_k<\infty) \ \leqslant\ \frac{3}{4}\,,
\end{align*}
where we used that $\beta_{T_k}\geqslant \beta_{t_A} \geqslant b_A$ for all $k\in\N$. This implies that a.s. there exists $J\in\mathbb{N}$ such that $S_J=\infty$. On $\Omega_{A}$, $T_{J}$ is finite and $\sup_{t\geqslant T_{J}} H(Z_t)\leqslant K_A$, which concludes.
\end{proof}

We can now finally prove  Theorem~\ref{Th}.

\begin{proof}[Proof of Theorem~\ref{Th}]
Let $\varepsilon>0$ and, thanks to Proposition~\ref{prop:born{\'e}}, let $K>1$ be such that $\mathbb{P}\left(\sup_{t\geqslant0} H(Z_t)\geqslant K\right)\leqslant\varepsilon$. Using Proposition~\ref{prop:CVtore}, for $t\geqslant 0$,
\begin{align*}
\mathbb{P}(H(Z_t)>\delta) &\leqslant \varepsilon + \mathbb{P}\left(\sup_{t\geqslant 0} H(Z_t)\leqslant K ; H(Z_t)>\delta\right) \\ &=\varepsilon + \mathbb{P}\left(\sup_{t\geqslant 0} H_K\left(Z^K_t\right)\leqslant K ; H_K\left(Z^K_t\right)>\delta\right) \\ &\leqslant \varepsilon + \mathbb{P}\left(H_K\left(Z^K_t\right)>\delta\right) \underset{t \rightarrow+\infty}\longrightarrow \varepsilon.
\end{align*}
As a consequence, $\limsup_{t\rightarrow\infty}\mathbb{P}\left(H\left(Z_t\right)>\delta\right)\leqslant\varepsilon$ for all $\varepsilon>0$, which concludes. 
\end{proof}

\subsection{Full process on a compact space}\label{fullprocess}

One of the main point of the proof of Theorem~\ref{non} is essentially to get something of the form 
\[\forall \delta>0,\qquad \mathbb P_{z_0} \po \sup_{t\geqslant 0} H(Z_t)< c+\delta\pf >0 \]
for suitable initial conditions $z_0$. 
First, let us highlight some of the difficulties in order to motivate the rest of this section. Refining the proof of Proposition~\ref{boundedness}, we would obtain a similar result but with $c$ replaced by $4c$. The factor $4$ is due to two things. First, in Lemma~\ref{integ}, we prove a bound of order $1/t^2$ while the proof of Proposition~\ref{boundedness} in fact only requires an integrable bound, i.e. $1/t^{1+\delta}$ for any $\delta>0$ is enough. The second factor 2 is lost when the Cauchy-Schwartz inequality is used in \eqref{eq:CS}. To solve this, the Cauchy-Schwartz inequality has to be replaced by the H{\"o}lder inequality, which means the $L^2$ estimate  of Proposition~\ref{evol0} is not sufficient and we need $L^p$ estimates for all $p>2$, or even better, $L^{\infty}$ estimates \pierre{(besides, in \cite{Pierro}, the convergence is stated in relative entropy, which is weaker than $L^p$ for any $p>1$, which is why we said in the introduction that these results were not sufficient to conclude in the fast cooling case)}. In fact in the elliptic case the proof of \cite{HoStKu} relies on $L^\infty$ bounds. In order to get such bounds in an hypocoercive case, we will work with $H^k$-Sobolev norms for $k\geqslant 1$, as in the work \cite{Hyp}  of Zhang, and then use Sobolev embeddings. This should be done with a correct control of the dependency in time of the constants, and to do so it is convenient to work with a process (both position and velocity) in a compact manifold. This is done by replacing the Hamiltonian $H(x,y)=U(x)+|y|^2/2$ by $H_K(x,y)=U^K(x)+W(y)$ where $W$ is some periodic potential with $W(y)=|y|^2/2$ below some threshold. This raises an issue in the dissipation of the hypocoercive modified entropy. Indeed, in the modified $H^1$-norm of Villani (and similarly in the modified $H^k$-norm of Zhang), the key point is that the missing coercivity in $x$ is recovered through a term
\[ \na_xh\cdot \left[\na_y,y\na_x\right]h = |\na_x h|^2\]
where $[A,B]=AB-BA$ stands for  the commutator of $A$ and $B$. When $H$ is replaced by $H_K$, this term becomes
 \[ \na_xh \cdot \left[\na_y,\na W(y)\na_x\right]h = \na_xh \cdot \na^2 W \na_x h \]
  which is negative at maxima of $W$. For this reason, and since anyway we are not really interested in the process above some energy threshold, we will add some Brownian noise in the position variable  where $W(y)\neq |y|^2/2$.
  
As a conclusion, for these reasons, in this section, we consider a process $Z^K=(X^K,Y^K)$ on a periodic torus $M_K^2$  solution of 
\begin{align}
  \left\{
      \begin{aligned}\label{eq2}
        &\dd X^K_t = \na W(Y^K_t)\dd t -\sigma(Y_t)\na U^K(X_t^K)\dd t + \sqrt{2\sigma(Y_t^K)\beta_t^{-1}}\dd \tilde B_t\\    
        &\dd Y^K_t = -\nabla U^K(X_t^K)\dd t - \gamma_t \na W(Y^K_t)\dd t + \sqrt{2\gamma_t\beta_t^{-1}}\dd B_t
      \end{aligned}
    \right.
\end{align}
where $B$ is as in  \eqref{eq}, $\tilde B$ is another independent $d$-dimensional Brownian motion, and $\sigma,W$ and $U^K$ are $2L_K$-periodic non-negative functions on $\R^d$ (with $M_K=(\R/2L_K\Z)^d$). \lj{The term $\sigma\na U^K$ has been added so that, for fixed $\gamma,\beta$, this new process admits the explicit stationary measure
\[
\mu_{\beta}^K \propto e^{-\beta H_K(x,y)}\dd x\dd y
\]
that satisfies a Poincar\'e inequality (see Section~\ref{hypo2}).} Here and in all this section, when there is no ambiguity we identify $2L_K$-periodic functions on $\R^d$ and functions on $M_K$. Let us now give the precise definition of $L_K,\sigma,W$ and $U^K$.

\medskip

\begin{figure}
    \centering
    \includegraphics[scale=0.33]{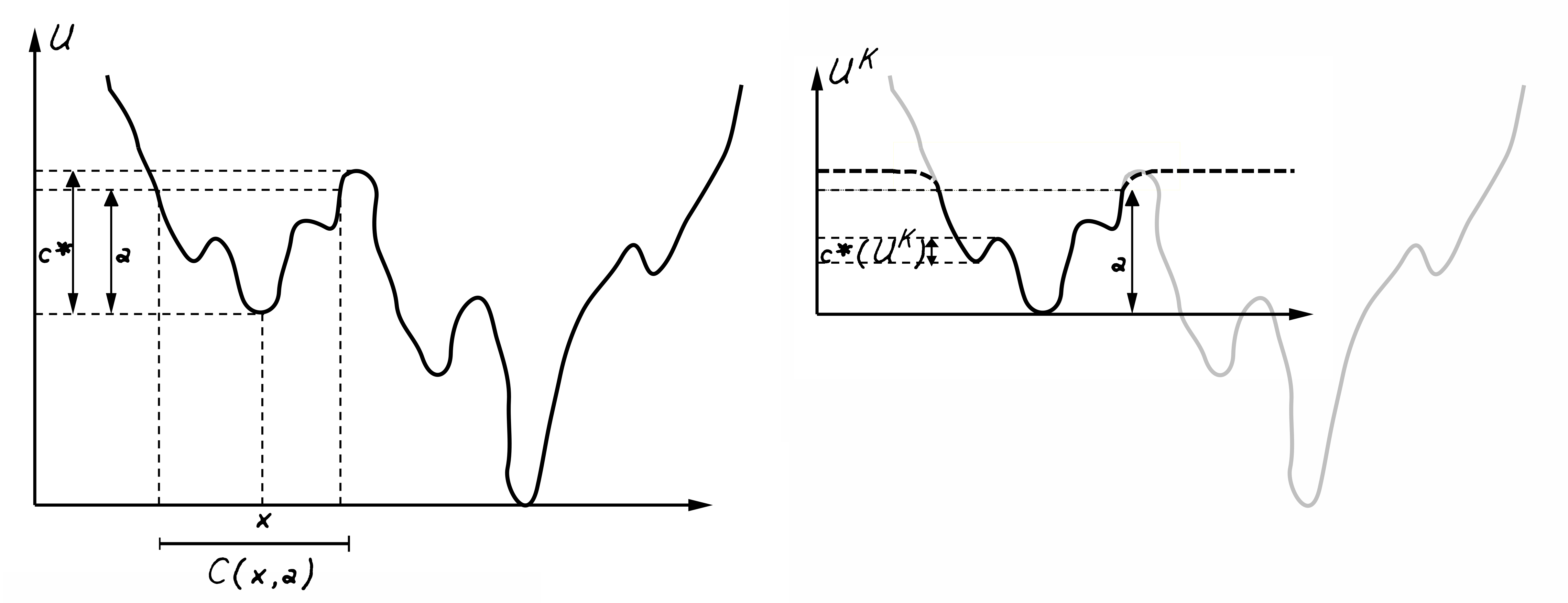}
    \caption{Construction of $U^K$ from $U$, with $\tilde x$ chosen so that $c^*(U)=\mathrm{D}(\tilde x)$.}
    \label{imageU}
    \label{fig:my_label}
\end{figure}

In all this section we consider  fixed $U\in \mathcal{C}^\infty(\R^d)$, $\beta$, $\gamma$, $\tilde x\in\R^d$ and $a>0$ satisfying Assumption~\ref{hyp2}. Setting $K=a+1$, similarly to Section~\ref{posicompset} we fix some $L_K>\sqrt{2K}+2$ large enough so that $\mathrm{C}(\tilde x,a) \subset [-L_K+1,L_K-1]^d$. We consider a non-negative $U^K\in \mathcal C^\infty(M_K)$ with $c^*(U^K)<c$ and such that, seen as a periodic function on $\R^d$, $U^K(x)=U(x)-U(\tilde x)$ for all $x\in\mathrm{C}(\tilde x,a)$. Such a function exists: indeed, since $\tilde c$ is continuous, $\mathrm{C}(\tilde x,a)$ is a compact set, and then, under Assumption~\ref{hyp2}, $\sup\{\tilde c(y,\tilde x),y\in\mathrm{C}(\tilde x,a)\}<c$. We can thus choose $\delta\in(0,a-c)$ small enough so that $\sup\{\tilde c(y,\tilde x) \lj{:} y\in\mathrm{C}(\tilde x,a)\}+2\delta<c$ and, using that the boundary of $\mathrm{C}(\tilde x,a)$ is in $\{U=U(\tilde x)+a\}$, we let $U^K$ be any $\mathcal{C}^\infty$ potential on $[-L_K,L_K]^d$ with $U^K(x)=U(x)-U(\tilde x)$ on $\mathcal C(\tilde x,a)$, $U^K(x)\in[a-\delta,a+\delta]$ for $x\in [-L_K,L_K]^d\setminus \mathrm{C}(\tilde x,a)$ with $U^K(x)=a$ for all $x\in[-L_K,L_K]^d\setminus[-L_K+1/2,L_K+1/2]$ so that we can extend it to a   $2L_K$-periodic $\mathcal C^ \infty$ potential on $\R^d$.  The fact that $c^*(U^K)<c$ is straightforward since a path from any $y\in[-L_K,L_K]^d$ to $\tilde x$ can be obtained by a straight line from $y$ to some $y_*\in\mathrm{C}(\tilde x,a)$ with $\tilde c(y,y_*)\leqslant 2\delta$ and then a path to $\tilde x$, where Assumption~\ref{hyp2} is used. See Figure~\ref{fig:my_label} for an illustration of the construction of $U^K$.

 Concerning the new potential $W$ for the velocity, first, write $n=\sqrt{2K}+1$ and $m=\sqrt{2K}+2$. Then, fix some $2L_K$-periodic $W_1\in\mathcal C^\infty(\R, \R)$ such that for all $s\in\left[-m,m\right]$, $W_1(s)=s^2/2$, $W_1$ is symmetric on $[-L_K,L_K]$, and increasing on $\left[0,L_K\right]$. For $y\in \left[-L_K,L_K\right]^d$, we set $W(y)=\sum_{i=1}^d W_1(y_i)$.  

Similarly, fix some $2L_K$-periodic  non-negative $\sigma^0\in \mathcal C^\infty(\R^d,\R_+)$ with   $\sigma^0(y)=0$ for $y\in[-n,n]^d$, $\sigma^0(y) = \sigma_*^0 $ for $y\in [-L_K,L_K]^d\setminus [-m,m]^d$ for some constant $\sigma_*^0>0$, and   $|\nabla \sigma^0|^2\leqslant C\sigma^0$ for some $C>0$.  Set  $\sigma= r \sigma^0$ where $r>0$ is chosen small enough that $\|\na\sigma\|_{\infty} \|\na U^K\|_{\infty}\leqslant\frac{1}{2}$ and $|\na \sigma|^2 \leqslant   \sigma$.

 The useful properties of $W$ and $\sigma$ can be summarized as follows :

\begin{lemma}\label{hypo:compact1} 
 
\begin{itemize}
\item Seeing $W$ as a periodic function from $\R^d$ to $\R$, there exists $m\in (\sqrt{2K}, L_K)$ such that $W(y)=|y|^2/2$ for all $y\in\left[-m,m\right]^d$.
\item On $M_K$, $W$ has a unique local minimum, which is global. 
\item There exist $n\in (\sqrt{2K},m)$ and $\sigma_{\ast}>0$ such that, seeing $\sigma$ as a periodic function from $\R^d$ to $\R_+$,  $\sigma(y)=\sigma_{\ast}$ for all $y\in\left[-L_K,L_K\right]^d \backslash\left[-m,m\right]^d$   and $\sigma(y) = 0$ for all $y\in\left[-n,n\right]^d$.
\item $\sigma$ satisfies $\|\na\sigma\|_{\infty} \|\na U^K\|_{\infty}\leqslant1/2$ and  $ |\na\sigma(y)|^2\leqslant \sigma(y)$ for all $y\in M_K$.
\end{itemize}
\end{lemma}

We write $H_K(x,y) = U^K(x) + W(y)$,
\[ \mathcal{M} = \left\{ y\in M_K;\na^2 W(y) = I_d \right\} \]
and notice that, by construction, for $y\in M_K \setminus\mathcal{M}$, $\sigma(y) = \sigma_*$. Moreover, if $|y|^2/2\leqslant K$ then $W(y)=|y|^2/2$ and $\sigma(y)=0$.

By construction of $U^K,W$ and $\sigma$, if we consider $Z$ and $ Z^K$ the solutions respectively of \eqref{eq} and \eqref{eq2} with the same initial condition in $\mathrm{C}(\tilde x,c)$ and the same Brownian motion $B$, then $H(Z_t)=H_K(Z_t^K)$ for all  $t\leqslant \tau := \inf\{s\geqslant0,\ H(Z_s) \geqslant a\}$. In particular, for $\delta>0$ small enough so that $c+\delta<\inf\{U^K(x),x\in[-L_K,L_K]^d \setminus \mathrm{C}(\tilde x,a)\}$,
\begin{equation}\label{eq:HKZK}
    \{H_K(Z_t^K) <c+\delta\ \forall t\geqslant 0\} \subset \{X_t \in \mathrm{C}(\tilde x,c+\delta) \ \forall t\geqslant 0\}  \,,
\end{equation} 
which means we are lead to prove that the left hand side has a positive probability.



Denote by $f_t^K$ the law of the solution of \eqref{eq2},  write $\mu^K_{\beta}={\mathcal{Z}_{\beta}^K}^{-1}e^{-\beta H_K(z)}\dd z$ where $\mathcal{Z}_{\beta}^K$ makes $\mu^K_{\beta}$ a probability measure, and let $h_t^K = f_t^K/\mu^K_{\beta_t}$. Similarly to Section~\ref{Loca}, the key point is the following estimate, proven in Section~\ref{hypo2}.

\begin{proposition}\label{evol1}
Let $f_0^K\in\mathcal{C}^{\infty}(M_K\times M_K)$. Under Assumption~\ref{hyp2}, $t\mapsto \|h_t^K\|_\infty$ is bounded. 
\end{proposition}

\begin{lemma}
Under Assumption~\ref{hyp2},  for all probability density $f_0^K\in\mathcal{C}^{\infty}(M_K\times M_K)$ and $\delta>0$, there exists $C>0$ such that for all $ t\geqslant 0$, 
\[ \mathbb{P}_{f_0^K}\left(H_K\left(Z^K_t\right) \geqslant c + \delta\right) \leqslant \frac{C}{(1+t)^{1+\frac{\delta}{2c}}}. \]
\end{lemma}

\begin{proof}
Applying  Proposition~\ref{evol1}, for $t\geqslant 0$
\begin{align*}
\mathbb{P}\left(H_K\left(Z^K_t\right) \geqslant c + \delta\right)
&= \int_{\left\{H_K \geqslant c + \delta\right\}} h_t^K \dd \mu^K_{\beta_t} \\
&\leqslant \sup_{s\geqslant 0}\|h_s^K\|_\infty \mu^K_{\beta_t}\left(H_K \geqslant c + \delta\right) \leqslant   \frac{C}{(1+t)^{(1+\frac{\delta}{2c})}}
\end{align*}
for some $C>0$, where we used Lemma~\ref{conveq} with $\alpha=\delta/2$.
\end{proof}

\begin{lemma}\label{opt}
 Under Assumption~\ref{hyp2},  for all probability density $f_0^K\in\mathcal{C}^{\infty}(M_K\times M_K)$ and $\delta>0$, there exists $t_b>1$ such that  \[\mathbb P_{f_0^K}\po \sup_{t\geqslant t_b}H_K(Z^K_t)<  c+\delta\pf>0.\]
\end{lemma}

\begin{proof}
It is enough to show the result for $\delta$ satisfying $c+\delta<a$.
Let $\B$ the event $\left\{H(Z_{t_b})<c\right\}$.
Let $\psi_K:\R\mapsto \left[0,1\right]$ be a $\C^{\infty}_c$ function equal to zero outside of $\left[ c+\frac{\delta}{2} , c + \delta +1\right]$ such that $\psi_K<1$ on $\left[c+\frac{\delta}{2},c+\delta\right[$ and $\psi_K(c+\delta)=1$ and write $\Psi_K= \psi_K\circ H_K$. We have :
\[ \left\{\sup_{t\geqslant t_b}H\left(Z^K_t\right) < c+\delta\right\} = \left\{\sup_{t\geqslant t_b}\Psi_K\left(Z^K_t\right) <1 \right\}.\]
There exists some constant $C>0$ such that $|L\Psi_{K}|\leqslant C(1+\beta_t)\mathbbm{1}_{\left\{H_{K}\geqslant c+\frac{\delta}{2}\right\}}$, and from Ito's formula we can write :
\[\Phi\left(Z_{t+t_b}^{K}\right)\mathbbm{1}_{\B}=\left(M_t + R_t\right)\mathbbm{1}_{\B}  \] where $R_t=\int_{t_b}^t L\Psi_{K}\left(Z^K_s\right)\dd s$, and if $t_b\geqslant t_0$:

\begin{align*}
\E\left(\sup_{t\geqslant t_b} |R_t|\mathbbm{1}_{\B}\right) &\leqslant C\int_{t_b}^{\infty}(1+\beta_s) \ec{\mathbbm{1}_{H^{K}\left(Z^K_t\right)\geqslant c+\frac{\delta}{2}}}{\B} \dd s \\ &\leqslant C\int_{t_b}^{\infty}(1+\beta_s) \mathbb{P}\left(H_K(Z^{K}_{s})\geqslant c+\frac{\delta}{2}|\B\right) \dd s \\ &\leqslant C\int_{t_b}^{\infty} \frac{1+\ln(1+t)}{(1+t)^{1+\frac{\delta}{8c}}} \dd s
\end{align*} 
where the constant $C$ depends only on $U^K$ and its derivative, but not on $t_b$. Since $(1+\ln(1+t))/(1+t)^{1+\frac{\delta}{8c}}$ is integrable, we can take $t_b$ great enough so that the event $E =\left\{\sup_{t\geqslant t_b} |R_t|\leqslant \frac{1}{10}\right\}$ has probability at least $\frac{3}{4}$.
On $E\cap\B$, $M_t$ takes value in $\left[-\frac{1}{10},\frac{11}{10}\right]$ because $0\leqslant\psi_A\leqslant1$. Using Doob's up-crossing inequality as in the proof of Lemma~\ref{boundedness}, we get that the probability of $\Psi_K$ going to 1 knowing $\B$ is less then $\frac{1}{2}$. We conclude by :
\begin{multline*}
\mathbb{P}\left(\sup_{t\geqslant 0}H_{K_A}(Z^{K_A}_t)< c+\delta\right) \geqslant \mathbb{P}(\B)\times\proc{\sup_{t\geqslant 0}H_{K_A}(Z^{K_A}_t)< c+\delta }{\B} >0.
\end{multline*}
\end{proof}

\subsection{Non-convergence with fast cooling schedules}


We are now ready to prove Theorem~\ref{non}. In  this section, Assumption~\ref{hyp2} is enforced and we use the definitions and notations of Section~\ref{fullprocess}. 




We start with a result on the position of the process for small times, as well as a Doeblin-like condition, which will be proven in Section~\ref{sec:regular} :
\begin{lemma}\label{acces}
For all $t,\delta>0$,  write :
\[\mathcal{B}_{t} = \left\{X_s\in \mathrm C(\tilde x,c+\delta)\ \forall s\in\left[0,t\right]\right\}.\]
Then, for all $z_0=(x_0,y_0)$ with $x_0\in \mathrm C(\tilde x,c)$, $\mathbb{P}_{z_0}\left(\mathcal B_{t}\right)>0$, 
and more precisely for all compact set $\mathcal K$ included in the interior of $\mathrm C(\tilde x,c+\delta)\times \R^d$, there exists $\varepsilon>0$ such that
\[\mathbb P_{z_0}\po\mathcal B_t,\ Z_t\in\cdot\pf\geqslant \varepsilon \ell\po \cdot\cap \mathcal K\pf,\]
where $\ell$ stands for the Lebesgue measure.

\end{lemma}



\begin{proof}[Proof of Theorem~\ref{non}]
By conditioning on the initial condition, it is sufficient to prove the result with a fixed initial condition $z_0=(x_0,y_0)$ with $x_0\in \mathrm{C}(\tilde x,c)$. Moreover it is sufficient to prove the result for $\delta>0$ small enough. We consider a fixed $\delta>0$ such that $c+\delta<\inf\{U^K(x),x\in[-L_K,L_K]^d \setminus \mathrm{C}(\tilde x,a)\}\leqslant a$, which means \eqref{eq:HKZK} holds and $\{U^K \leqslant c+\delta\}=\mathrm{C}(\tilde x,c+\delta)$ (seeing $\{U^K \leqslant c+\delta\}$ as a subset of $[-L_K,L_K]^d$).

Let $f_0^K\in\mathcal{C}^{\infty}(M_K\times M_K)$ a probability density and $t_b>0$ as in Lemma~\ref{opt} applied with $\delta/2$ instead of $\delta$ such that:\[\mathbb P_{f_0^K}\left(\sup_{t\geqslant t_b} H_K(Z_t^K)\leqslant c+\delta/2\right)>0.\]
In other words, denoting by $\tilde f$ the law at time $t_b$ of the process solution to equation~\eqref{eq2} with initial condition $f_0^K$ and conditioned to $\{H^K(Z_{t_b})\leqslant c+\delta/2\}$, and $(Z^{K,t_b})_{t\geqslant 0}$ the solution to equation~\eqref{eq2} with initial condition $\tilde{f}$ at time $t_b$, 
\[\mathbb{P}_{\tilde{f}}\left(\sup_{t\geqslant 0}H_K(Z^{K,t_b}_t) \leqslant c+\delta/2\right)>0.\]
Let $\mathcal{B}_{t_b}$ be as in Lemma~\ref{acces}. From Lemma~\ref{acces} and the fact $\tilde f$ has a bounded density on its support $\{H^K \leqslant c+\delta/2\}$ (which we see as a subset of $[-L_K,L_K]^{2d}$), 
 there exists $\varepsilon>0$ such that
\[\mathbb{P}_{z_0}\left(B_{t_b},\ Z_{t_b}\in\cdot\right)\geqslant \varepsilon\tilde f.\]
Finally, thanks to \eqref{eq:HKZK},  we conclude that
\begin{align*}
\mathbb{P}_{z_0}\left(  X_t \in \mathrm{C}(\tilde x,c+\delta)\ \forall t\geqslant 0\right) & \geqslant \mathbb{P}_{z_0}\po\mathcal B_{t_b},\   X_t \in \mathrm{C}(\tilde x,c+\delta)\ \forall t\geqslant t_b\pf \\
& \geqslant \varepsilon   \mathbb{P}_{\tilde{f}}\left(\sup_{t\geqslant 0}H(Z^{K,t_b}_t)<c+\delta\right) >0.
\end{align*}
\end{proof}

\section{Auxiliary results}\label{sec:auxiliary}

\subsection{Uniform energy bounds}\label{EnergyBounds}

This section is dedicated to the proof of Lemma~\ref{energybounds}. First, we consider a family of approximation functions $ \eta_m \in \C_c^{\infty}$ for $m\geqslant 1$ in order to justify some PDE computations below. Let
\[\Phi(s) = \left\{
      \begin{array}{ll}
       e^{\frac{1}{s^2-1}}/\int e^{\frac{1}{u^2-1}}\dd u  & \text{for } s\in\left(-1,1\right) \\    
         0 & \text{ for }s\in \R\setminus(-1,1)
      \end{array}
    \right.\]	
and, for $m\geqslant 1$, $\Phi_m(s) = \Phi(s/m)/m$, $\nu_m = \mathbbm{1}_{\left(-\infty,m^2\right]}\ast\Phi_m$ \lj{where $\ast$ denotes the convolution,} and finally, for $z\in\R^{2d}$,
\[ \eta_m\left(z\right) = \nu_m\left(\ln\left(H(z)+1\right)\right)\,.\]
		
\begin{proposition}\label{approx1}
Assume $U\rightarrow +\infty$ at infinity and $\na U$ is bounded.
\begin{itemize}
\item For all $ m\geqslant 1$, $\eta_m\in\C^{\infty}(\R^d)$, has a compact support, and satisfies $0\leqslant \eta_m \leqslant 1$.
\item For all $z\in\R^{2d}$,  $\eta_m(z) \rightarrow 1$ as $m\rightarrow\infty$.
\item There exists $C>0$ such that $L_t\eta_m\leqslant C\beta_t/m$ and $|\na\eta_m|\leqslant C\beta_t/m$ for all  $m\geqslant 1$.
\end{itemize}
\end{proposition}

\begin{proof}
This is \cite[Lemma 4]{GenLan}.
\end{proof}

\begin{proof}[Proof of Lemma~\ref{energybounds}]
We first show the result when $\na U$ and its derivative are bounded, the general case being then obtained by approximating  $U$.

Hence, suppose for now that $\na U$ and its derivative are bounded. In this case, \lj{it can be shown that} the law of the process at time $t$ admits a bounded density $f_t$ such that $(t,z)\mapsto f_t(z) \in \mathcal{C}^{\infty}\left(\R_+\times\R^{2d}\right) $\lj{, see the proof of Lemma~\ref{regu}}. Define $g_t(z)=f_t(z)e^{\beta_t H(z)}$ and $u(t)=\E\left(H(Z_t)\right)$. Since $f$ and $U$ are smooth, so is $g$.
Consider
\[ N(t) = \int_{\R^{2d}} g_t(z)\ln\left(1+g_t(z)\right)e^{-\beta_tH(z)}\dd z\,. \] 
From the inequality $\ln(1+x) \leqslant \ln(x) + 1/x$ for all $x>0$,  
\[N(t) \leqslant \int_{\R^{2d}} f_t(z)\ln(f_t(z)) \dd x\dd y + \beta_t\E(H(Z_t)) + \int_{\R^{2d}} e^{-\beta_tH(z)} \dd z.\]
Since $f_t$ is bounded on $\left[0,T\right]\times\R^{2d}$, and is in $L^1(\dd z)$ for all $t\geqslant 0$, $t\mapsto \int_{\R^{2d}} f_t \ln(f_t ) $ is locally bounded and so is $N(t)$. In order to differentiate $N$, we introduce for $m\geqslant 1$ the approximation
\[ N_m(t) = \int_{\R^{2d}} \eta_m(z)g_t(z)\ln\left(1+g_t(z)\right)e^{-\beta_tH(z)}\dd z .\]
Integrating by parts, we see that the   dual of the generator $L_t$ in $L^2(e^{-\beta_t H})$ is
\begin{equation}\label{eq:L*}
L^{\ast}_{t} = -y\cdot \nabla_x + (-\gamma_t y + \nabla_xU)\cdot \nabla_y + \gamma_t\beta_t^{-1}\Delta_y
\end{equation}
Since $f_t$ solves $\partial_tf_t=L^*_t\left(f_te^{\beta_tH}\right)e^{-\beta_tH}$, $g_t$ solves $\partial_tg_t = \beta_t'Hg_t + L^*_tg_t$. Using that $\eta_m$ is compactly supported for all $m\geqslant 1$, we can differentiate  
\begin{align*}
N_m'(t) &= \int_{\R^{2d}} \eta_m\bigg(-\beta_t'Hg_t\ln(1+g_t)e^{-\beta_t H}  + \partial_tg_t\ln(1+g_t)e^{-\beta_t H} + \frac{g_t}{1+g_t}\partial_tg_te^{-\beta_tH}\bigg)   \\ &= \int_{\R^{2d}} \eta_mL^*_tg_t\left(\ln(1+g_t) + \frac{g_t}{1+g_t}\right)e^{-\beta_tH}  +\int_{\R^{2d}} \eta_m\frac{g_t^2}{1+g_t}\beta_t'He^{-\beta_tH}  .
\end{align*}
Now, using that $g_t\geqslant0$,
\[ \int_{\R^{2d}} \eta_m\frac{g_t^2}{1+g_t}\beta_t'He^{-\beta_tH}   \leqslant \beta_t'\int_{\R^{2d}} f_tH   = \beta_t'u(t).\]
Consider the carr{\'e} du champ operator $\Gamma_t$ associated to $L_t^*$ given by
\begin{equation}\label{gamma}
\Gamma_t(g_1,g_2) := \frac12 \po L_t^*(g_1g_2) - g_2 L_t^*(g_1)-g_1L_t^*(g_2)\pf = \gamma_t\beta_t^{-1} \na_y g_1 \cdot \na_y g_2
\end{equation}
for smooth $g_1,g_2$, and $\Gamma_t(g):=\Gamma_t(g,g)$. 
Using that $e^{-\beta_t H} $ is invariant for  $L_t^*$ so that $\int L_t^* g e^{-\beta_t H} = 0$ for all smooth $g$ and the diffusion property
\[L_t^*(\phi(g)) = \phi'(g)L_t^*(g) + \phi''(g)\Gamma_t(g)\]
for any smooth $\phi$, we get
\begin{align*}
\int_{\R^{2d}} &\eta_mL^*_tg_t\left(\ln(1+g_t) + \frac{g_t}{1+g_t}\right)e^{-\beta_tH}  \\&= \int_{\R^{2d}} \eta_mL^*_tg_t\left(\ln(1+g_t) + \frac{g_t}{1+g_t}\right)e^{-\beta_tH}  - \int_{\R^{2d}} L^*_t(\eta_mg_t\ln(1+g_t)) e^{-\beta_t H}  \\&= -\int_{\R^d\times\R^d} \eta_m\Gamma_t(g_t)\left( \frac{1}{1+g_t} + \frac{1}{(1+g_t)^2}\right)e^{-\beta_tH} - \int_{\R^{2d}} L^*_t(\eta_m)g_t\ln(1+g_t)e^{-\beta_tH}  \\&\ \ \ \ \ \ \ - 2\int_{\R^{2d}} \Gamma_t(\eta_m,g_t\ln(1+g_t)) e^{-\beta_tH}\,.
\end{align*}
The first term is negative (since $\eta_m \Gamma_t(g_t)\geqslant 0$) and the others are bounded as follows:
\begin{align*}
 - \int_{\R^d\times\R^d} L^*_t(\eta_m)g_t\ln(1+g_t)e^{-\beta_tH} - 2 \int_{\R^d\times\R^d}& \Gamma(\eta_m,g_t\ln(1+g_t)) e^{-\beta_tH}  \\
 &  =  \int_{\R^d\times\R^d}  \eta_mL^*_t(g_t\ln(1+g_t)) e^{-\beta_tH} 
 \\ &= \int_{\R^d\times\R^d} L_t(\eta_m)g_t\ln(1+g_t)e^{-\beta_tH}\dd x\dd y \\ &\leqslant \frac{C\beta_t}{m} N(t)
\end{align*}
thanks to Proposition~\ref{approx1}. 
We conclude that for all $m\geqslant 1$, $0\leqslant s \leqslant t$ :
\[ N_m(t) - N_m(s) \leqslant \int_s^t \beta_r'u(r) + \frac{C\beta_r}{m}N(r)  \dd r. \]
By the monotone convergence theorem, $N_m(t) \rightarrow N(t)$ as $m\rightarrow\infty$, hence :
\[ N(t) - N(s) \leqslant \int_s^t \beta_r'u(r) \dd r. \]
On the other hand, from the variational formula for the entropy:
\[
\int_{\R^{2d}} f_t\ln f_t     = \max\left\{\int_{\R^{2d}} f_t\ln g ;\ g:\R^{2d}\mapsto\R_+,\ \int_{\R^{2d}}g = 1 \right\}\,,
\]
we get with $g_0 = e^{-\alpha_0H }/ \mathcal{Z}_{\alpha_0}$, where $ \mathcal{Z}_{\alpha_0} = \int e^{-\alpha_0H }$,
\[ N(t)   \geqslant \beta_tu(t) + \int_{\R^{2d}} f_t\ln f_t   
  \geqslant \beta_tu(t) + \int_{\R^{2d}} f_t\ln g_0  = (\beta_t-\alpha_0)u(t) - \ln(\mathcal{Z}_{\alpha_0}).\]
As a consequence, writing $\phi(t)=N(t) + \ln(\mathcal{Z}_{\alpha_0})$ for $t\geqslant 0$, 
\[\phi(t)\leqslant \phi(0) + \int_0^t \frac{\beta'_s}{\beta_s-\alpha_0}\phi(s)\dd s .\]
Then Gronwall's lemma gives $\phi(t)\leqslant \phi(0)\frac{\beta_t - \alpha_0}{\beta_0-\alpha_0}$ and using again that $u(t) \leqslant \phi(t) /(\beta_t-\alpha_0)$ concludes the proof in the case where $\na U$ and all its derivatives are bounded.

Let us now consider the general case, without any assumption on $\na U$. For $n$ large enough, we \lj{fix some $U_n\in\C^{\infty}(\R^d)$ equal to $U$ on $B(0,n)$, to $|x|$ outside of $B(0,n+1)$ and such that for all $x\in\R^d$, $U_n(x)\geqslant \min(U(x),|x|)-1$}. Let $(Z^n_t)_{t\geqslant 0}=(X^n_t,Y_t^n)_{t\geqslant 0}$ be the diffusion defined by Equation~\eqref{eq} where $U$ is replaced by $U_n$ and starting from the same initial distribution $f_0$. By design, $Z_{t}$ and $Z_{t}^n$ are equal up to the time $\tau_n = \inf\left\{t\geqslant 0; |X_{t}|\geqslant n\right\}$. As $Z$ does not explode in finite time, $\lim_n\tau_n=\infty$ and, for all $t\geqslant 0$, $\lim_nH_n(Z^n_{t})=H(Z_{t})$ almost surely, where $H_n(x,y)=U_n(x)+|y|^2/2$. By Fatou's lemma, for all $t\geqslant 0$,
\begin{multline*}
\E(H(Z_t)) = \E\left(\liminf_{n\rightarrow\infty} H_n(Z^n_t)\right) \leqslant \liminf_{n\rightarrow\infty} \E(H_n(Z^n_t)) \leqslant \liminf_{n\rightarrow\infty} \frac{\kappa^{\beta_0,n}(f_0)+\ln(\mathcal{Z}^n_{\alpha_0})}{\beta_0-\alpha_0}
\end{multline*}
where $\kappa^{\beta_0,n}(f_0)$ is define as $\kappa^{\beta_0,n}(f_0)$ but with $U$ replaced by $U_n$. Now, since $f_0$ is compactly supported, $\kappa^{\beta_0,n}(f_0)$ is independent of $n$ for $n$ large enough. Finally, using that $e^{-\beta_0 U_n(x)} \leqslant e^{\alpha_0}(e^{-\alpha_0 U(x)}+e^{-\alpha_0 |x|})$,  we can apply the Dominated Convergence Theorem to get that $ \mathcal{Z}^n_{\alpha_0} \rightarrow  \mathcal{Z}_{\alpha_0}$  as $n\rightarrow \infty$, which concludes.

\end{proof}
\subsection{Small time regularisation}\label{sec:regular}


\begin{proof}[Proof of Lemma~\ref{changcondinit}]
	It is enough to show the first point for $\na U$ bounded. Indeed, for all $\tilde{U}$ equal to $U$ on $\left\{U\leqslant A+1\right\}$, and $(\tilde{Z})_{t\geqslant 0}$ the corresponding process, we have the equality 
	\[\left\{\sup_{t\leqslant t^*}H(Z_t)\leqslant A+1\right\} = \left\{\sup_{t\leqslant t^*}\lj{\tilde H}(\tilde{Z}_t)\leqslant A+1\right\}.\] 
	We only need a bound on $\sup_{t\leqslant t^*}|Y_t-y_0|$ because :
	\[ \sup_{t\leqslant t^*}|U(X_{t})-U(x_0)|\leqslant \|\na U\|_{\infty}t^*\sup_{t\leqslant t^*}|Y_t|.\]
From
	\[ Y_t = e^{-\int_0^t\gamma_s\dd s}\left( y_0 + \int_0^t \na U(X_s)e^{\int_0^s \gamma_u\dd u} \dd s + \int_0^t \sqrt{2\gamma_s\beta_s^{-1}} e^{\int_0^s\gamma_u\dd u}\dd B_s\right)\,, \]
	we get
	\[ \sup_{s\leqslant t} |Y_s| \leqslant |y_0| + \|\na U\|_{\infty}t + \sup_{s\leqslant t}|W_s|\,,\]
	where $W_t=e^{-\int_0^t\gamma_s\dd s}\int_0^t \sqrt{2\gamma_s\beta_s^{-1}} e^{\int_0^s\gamma_u\dd u}\dd B_s$. Since 	$(W_t)_{0\leqslant t\leqslant 1}$ is a $L^2$-martingale,   Doob's inequality implies  
	\[ \E\left(\sup_{0\leqslant s\leqslant t} |W_s|^2\right) \leqslant 4\E(W_t^2) = 4e^{-2\int_0^t\gamma_s\dd s}\int_0^t 2\gamma_s\beta_s^{-1} e^{2\int_0^s\gamma_u\dd u}\dd s \leqslant 8Lt. \]
	Now we can take $t^*_1$ such that $t^*_1\|\na U\|_{\infty}+\sup_{t\leqslant t^*_1}W_t\leqslant \frac{1}{8\sqrt{A}}$ with probability larger than $1-\varepsilon$. Thus, with probability at least $1-\varepsilon$, we have :
	\[ \sup_{t\leqslant t^*_1} |Y_t|^2 \leqslant |y_0|^2 + \frac{1}{4} + \frac{1}{32A}.\]
	Since $A>1$, this is less than $|y_0|^2+\frac{1}{2}$ and we write $t^*_2 = t^*_1\wedge (2\|\na U\|_{\infty}\sqrt{1+|y_0|^2})^{-1}$. 
	
	For the second point, fix some $t^*\leqslant t^*_2$ and let, for $t\leqslant t^*$, $(\bar{X}_t,\bar{Y}_t)$ be the solution of the system
	\begin{equation*}
		\left\{
		\begin{aligned}
			&\dd \bar{X}_t = \bar{Y}_t\dd t\\    
			&\dd \bar{Y}_t = \sqrt{2\gamma_t\beta_t^{-1}}\dd B_t\\
			&(X_0,Y_0)=(x_0,y_0),\\
		\end{aligned}
		\right.
	\end{equation*} 
	let
	\[N_t=-\int_0^t\frac{\sqrt{\beta_s}}{\sqrt{2\gamma_s}} \left(\na U(\bar{X}_s) + \gamma_s\bar{Y}_s \right)\dd B_s\]
	and $\mathbb{Q}=e^{N_{t^*}-\frac{1}{2}\left\langle N\right\rangle_{t^*}}\mathbb{P}$. 	Ito's formula gives for $t\leqslant t^*$ :
	\begin{multline}\label{Girsanov}
		\frac{\beta_t}{2\gamma_t}\bar{Y_t}\cdot(\na U(\bar{X}_t) + \frac{\gamma_t}{2}\bar{Y}_t) - \frac{\beta_0}{2\gamma_0}y_0\cdot (\na U(x_0) + \frac{\gamma_0}{2}y_0) \\= -N_t + \int_0^t \left( \frac{\beta_s}{2\gamma_s}\Delta U(\bar{X}_s)|\bar{Y_s}|^2 + \frac{\beta'_s}{2\gamma_s}\bar{Y_s}\cdot (\na U(\bar{X}_s) + \frac{\gamma_s}{2}\bar{Y}_s)-\frac{\beta_s\gamma'_s}{2\gamma_s^2}\bar{Y_s}\cdot \na U(\bar{X}_s) + \frac{\beta_s}{4} \right) \dd s.
	\end{multline}
	Using $\beta'\leqslant c$ and $\beta_t\leqslant \beta_0+\frac{t^*}{c}$, we then get $\mathbbm{1}_{\B}N_t\leqslant C(\beta_0+1)$ where $C$ is independent from $\beta_0$ and $\gamma$ and depends only on $\sup_{\left\{U\leqslant A+1\right\}}|\na U|$, $\sup_{\left\{U\leqslant A+1\right\}}|\Delta U|$, $\kappa$ and $c$.
	
	Girsanov's theorem yields that, under the change of probability $\mathbb{P}\rightarrow\mathbb{Q}$,  $(\bar{X}_t,\bar{Y}_t)_{t\leqslant t^*}$ is a solution to the original Equation~\eqref{eq} and, for all $\phi \geqslant 0$, $t\leqslant t^*$,
	\[ \E(\phi(X_t,Y_t)\mathbbm{1}_{\B})=\E(e^{N_{t^*}-\frac{1}{2}\left\langle N\right\rangle_{t^*}}\phi(\bar{X}_t,\bar{Y}_t)\mathbbm{1}_{\B})\leqslant Ce^{C^1_{A}(\beta_0+1)} \E(\phi(\bar{X}_t,\bar{Y}_t)\mathbbm{1}_{\B}) \,.\]
	 It only remains to show that $(\bar{X}_t,\bar{Y}_t)$ has a density bounded by some $e^{C\beta_0}$. As a  Gaussian process,		the density of $(\bar{X}_t,\bar{Y}_t)$ is bounded by $(2\pi)^{-d/2}/\det(Q_d)$ where $Q_d$ is the covariance matrix of the process at time $t^*$. By independence, $\det(Q_d)=\left(\det(Q_1)\right)^d$, and a straightforward computation yields
		\[ \det(Q_1) = 8\int_0^{t^*} \gamma_s\beta_s^{-1}\dd s \int_0^{t^*}(t^*-s)\int_0^s\gamma_u\beta_u^{-1}\dd u\dd s -4\left( \int_0^{t^*}\int_0^s \gamma_u\beta_u^{-1}\dd u\dd s \right)^2\,.\]
		Now, let $\varpi, \lambda>1$ be such that $\varpi\lambda^2 < 4/3$ and take $t^*$ small enough (depending only on $c$ and $\kappa$) so that uniformly in $\beta_0\geqslant 1$ 
		\[ \beta_0\leqslant \beta_{t^*} \leqslant \varpi\beta_0 \quad \text{and}\quad \lambda^{-1}\gamma_0\leqslant \gamma_{t^*} \leqslant\lambda\gamma_0.\]
This choice ensures 
		\[ \det(Q_1) \geqslant \kappa\beta_0^{-1}{t^*}^4\left( \frac{4}{3\varpi\lambda}-\lambda^2 \right) \]
		and thus the result.
\end{proof}

%

\begin{proof}[Proof of Lemma~\ref{acces}]
The first part follows from the controlability of the process, see \cite[proposition 5]{Pierro}. The second part follows from the fact the distribution of the process killed when it leaves $\mathrm C(\tilde x,c+\delta\}\times\R^d$ solves a parabolic hypoelliptic Dirichhet problem, hence has a continuous positive density.
In fact, in the time-homogeneous case, this is exactly  \cite[Theorem 2.20]{ramiletal}. In our time-inhomogeneous case, we can proceed as follows. First, for some $A'>\max(|y_0|,A)$ large enough, consider $\mathcal{D}$ the interior of $\mathrm C(\tilde x,c+\delta\}\times[-A',A']$ and $\tau=\inf\{s\geqslant0, Z_t\notin \mathcal D\}$. Then, for all $t>0$,
\[\mathbb P_{z_0}\po\mathcal B_t,\ Z_t\in\cdot\pf \geqslant \mathbb P_{z_0}\po\tau> t,\ Z_t\in\cdot\pf :=p_t^{\mathcal D}(z_0,\cdot)\,.\]
It is well-known that $p_t^{\mathcal D}$ solves a parabolic equation with generator $L_t$ on $\mathcal D$. Since $L_t$ is hypoelliptic, $p_t^{\mathcal D}(z_0,\cdot)$ has a continuous density. Finally, the coefficients of $L_t$ being smooth and bounded on $\mathcal D$, $p_t^{\mathcal D}(z_0,\cdot)$ being not identically zero and the process being controllable, we can use the strong maximum principle of \cite[Theorem 6.1]{StroockVaradhan} to deduce that $p_t^{\mathcal D}(z_0,\cdot)$ cannot take the value $0$ in $\mathcal D$. As a continuous positive function, $p_t^{\mathcal D}(z_0,\cdot)$ is thus lower bounded by a positive constant over any compact subset of $\mathcal D$, which concludes.

\end{proof}

\section{ \texorpdfstring{$L^2$}{L2}-hypocoercivity}\label{hypo1}

In this section we use the definitions and notations of Section~\ref{posicompset}

The hypocoercivity issue arises in the proof of Proposition~\ref{evol0} when computing the evolution of the $L^2$-norm of $h_t^K$. In the standard elliptic case, as in \cite{HoStKu}, one would simply differentiate this quantity and concludes with a Poincar{\'e} inequality of the form :\[ \int \po h -\int h\dd\mu\pf^2 \dd \mu \leqslant C\int \Gamma(h) \dd\mu \]
where $\Gamma$ is the carr{\'e} du champ associated to the process. However, in the kinetic case, as we saw  in \eqref{gamma}, $\Gamma_t(f) = \gamma_t\beta_t^{-1}|\na_yf|^2$, which means such an inequality cannot hold since, for non-constant functions of $x$, the left hand side is positive while the right hand side vanishes. For this reason, we work with a modified norm as in \cite{Vil}.

More precisely, at a formal level, the proof of Proposition~\ref{evol0} is the following: writing $\phi_{t}(h) = |(\na_x+\na_y)h|^2 + \sigma_t h^2$ with $\sigma_t=\frac{1}{2} + 2\sqrt{\gamma_t^{-1}\beta_t}(1+\|\na U^K\|_{\infty} + \gamma_t)^2$, we introduce
\[\tilde N(t) =   \int_{M_K\times\R^d} \phi_{t}\left(h_t^K-1 \right)\dd\mu^K_{\beta_t}\,,\qquad \tilde I(t) =   \int_{M_K\times\R^d}  \left|\na h_t^K\right|^2\dd\mu^K_{\beta_t}\,.\]
Differentiating $\tilde N$, one can (formally) check that
\[\tilde N'(t) \leqslant - \frac12 \tilde I(t) + C \beta_t' (1+\beta_t) \tilde N(t)\]
for some constant $C>0$, the definition of $\tilde N$ being motivated by the $-\tilde I$ term in this inequality. Using a Poincar\'e inequality for $\mu^K_{\beta_t}$ (with the full gradient rather than $\Gamma_t$) with a constant $\lambda(\beta_t)$ that scales as $1/t^{c^*/c}\gg \beta_t' \beta_t$, we get that $\tilde N'(t) \leqslant 0$ for $t$ large enough (or equivalently for all $t\geqslant0$ if $\beta_0$ is large enough), hence $\tilde N$ is bounded. We conclude the proof of Proposition~\ref{evol0} by bounding 
\[\|h_t\|_{L^2(\mu_{\beta_t})} \leqslant 1 + \|h_t-1\|_{L^2(\mu_{\beta_t})} \leqslant 1 + 2 \tilde N(t)\,.\]

In the remainder of this section, this formal proof is made rigorous and we give the details of the computations.

\pierre{
\begin{remark}\label{rem:err_truncation}
The proof of \cite{Pierro}  is based on a similar argument but, as mentioned in the introduction, and as noticed by the authors of \cite{GenLan}, it contains an error. The problem occurs when it comes to justify rigorously the derivation of $\tilde N$. In \cite{Pierro}, a compactly-supported truncation function is added within the integral. This leads to additionnal terms in $\partial_t \tilde N$. One of these terms is said to be non-positive in  \cite[Lemma 16]{Pierro}, which is false (there is a sign error). In \cite{GenLan},  the authors add a small elliptic term to the dynamics, use elliptic regularity results to justify the computation and then let the small ellipticity parameter vanish afterwards. In this section, we make a correct version of the argument of \cite{Pierro}, combining some bounds on the density of the process (Lemma~\ref{regu} below) and some moment estimates (Lemma~\ref{espexp} below). Also, notice that, by comparison with \cite{Pierro,GenLan}, we have already reduced the problem to a compact state space for the position $x$, the non-compact part of the dynamics only concerns the velocity.
\end{remark}
}

 First, we need a few preliminary lemmas. We start by stating the following Poincar{\'e} inequality (with the full gradient) :
\begin{proposition}\label{poin}
For $U^K \in \mathcal C^\infty(M_K,\R)$  there exists $\lambda:\R_+\rightarrow\R_+$ such that for all $f\in\C^{\infty}(M_K\times\R^d)$ with compact support :
\[ \lambda(\beta)\int_{M_K \times \R^d} \left(f-\int_{M_K\times\R^d} f\dd\mu_{\beta}^K \right)^2\dd\mu_{\beta}^K \leqslant \int_{M_K\times\R^d}|\nabla f|^2 d\mu^K_{\beta}\,, \]
and moreover
\[ \lim_{\beta\rightarrow\infty}\frac{1}{\beta}\ln(\lambda(\beta)) = -c^*(U^K) \]
where $c^*(U^K)$ is defined as $c^*$ with $U$ replaced by $U^K$.
\end{proposition}

\begin{proof}
	The fact that the Gibbs probability measure $\tilde{\mathcal Z}e^{-\beta U^K} $ satisfies a Poincar{\'e} inequality with a constant $\lambda(\beta)$ satisfying \[ \lim_{\beta\rightarrow\infty}\frac{1}{\beta}\ln(\lambda(\beta)) = -c^*(U^K) \] corresponds to \cite[Theorem 1.14]{HoStKu}. The Gaussian measure $\mathcal{N}(0,\beta I)$ also satisfies a Poincar{\'e} inequality with constant $\beta$, and we conclude with the tensorization property of the Poincar{\'e} inequality, see \cite[Proposition 4.3.1]{BakryGentilLedoux}.
\end{proof}

For $\mu$ a probability measure, $H^1(\mu)$ \lj{denotes} the usual Sobolev space of functions in $L^2(\mu)$ with derivative in $L^2(\mu)$.

\begin{proposition}\label{Vila}
There exists $C>0$ such that for all $ \beta>1$ and $g\in H^1\left(\mu_\beta^K\right)$, 
\[\int y^2g(x,y)^2 \mu_\beta^K(\dd x\dd y) \leqslant  C\left(\int g(x,y)^2\mu_\beta^K(\dd x\dd y) + \int |\na_yg(x,y)|^2\mu_\beta^K(\dd x\dd y)\right).\]
\end{proposition}

\lj{\begin{proof}
This is a particular case of \cite[Lemma A.18]{Vil}. \pierre{We give a short proof for completeness.} 
 Notice that $ye^{-\beta |y|^2/2}=-\beta^{-1}\na e^{-\beta |y|^2/2}$. Hence an integration by parts and Young's inequality yield, \pierre{for any $x$,}
\begin{align*}
\lefteqn{\int_{\R^d} y^2 g^2\pierre{(x,y)} e^{-\beta |y|^2/2} \dd y}\\
&= -\beta^{-1}\int_{\R^d}  g^2\pierre{(x,y)} y\cdot\na e^{-\beta |y|^2/2} \dd y\\ &= \beta^{-1}\int_{\R^d}  \na\cdot (g^2\pierre{(x,y)} y) e^{-\beta |y|^2/2} \dd y \\ &= d\beta^{-1}\int_{\R^d}  g^2\pierre{(x,y)} e^{-\beta |y|^2/2} \dd y + \beta^{-1}\int_{\R^d} 2g\na g\pierre{(x,y)}\cdot y e^{-\beta |y|^2/2} \dd y \\ &\leqslant d\beta^{-1}\int_{\R^d}  g^2\pierre{(x,y)} e^{-\beta |y|^2/2} \dd y + \beta^{-2}\int_{\R^d} 2|\na g\pierre{(x,y)}|^2 e^{-\beta |y|^2/2} \dd y \\ &\qquad + \frac{1}{2}\int_{\R^d} y^2 g^2\pierre{(x,y)} e^{-\beta |y|^2/2} \dd y,
\end{align*}
and thus 
\[\int_{\R^d} y^2 g^2\pierre{(x,y)} e^{-\beta |y|^2/2} \dd y\leqslant 2d\beta^{-1}\int_{\R^d}  g^2\pierre{(x,y)} e^{-\beta |y|^2/2} \dd y + 4\beta^{-2}\int_{\R^d} |\na g\pierre{(x,y)}|^2 e^{-\beta |y|^2/2} \dd y .\]
\pierre{Conclusion follows by integrating with respect to $x$.}
\end{proof}}

The  next  two lemmas will be used in forthcoming computations to justify that some quantities are finite and \lj{therefore allowing to interchange differentiation and integration}.

\begin{lemma}\label{espexp}
Fix some $K>0$ and any $c>0$. Then for all $\alpha>0$ and initial condition $f_0^K\in\mathcal{C}^{\infty}$ with compact support, there exists $b_{\alpha,K}$ such that if $\beta_0\geqslant b_{\alpha,K}$, then \[t\mapsto\E_{f_0^K}\left(e^{(\beta_t-\alpha) H_K\left(Z^K_t\right)}\right) \]
is finite and locally bounded.
\end{lemma}

\begin{proof}
Write $\tau_N=\inf\left\{t\geqslant0, H_K(Z^K_t)\geqslant N\right\}$ and $\phi_t(x,y)=e^{(\beta_t-\alpha)H_K(x,y)}$. For $\beta_0>\alpha$, we compute for $t\in\left[0,T\right]$ :
\begin{align*}
&\partial_t\phi_t + L_t\phi_t \\&=\bigg(\beta'_tH_K + (\beta_t-\alpha)y.\na U^K-\gamma_t\left(\beta_t-\alpha\right)y.\na_yH_K \\ &\qquad +\beta_t^{-1}\gamma_t\left(\left(\beta_t-\alpha\right)\Delta_yH_K + \left(\beta_t-\alpha\right)^2|\na_yH_K|^2 \right)\bigg)\phi_t \\& \leqslant \bigg( \beta'_t\left(\|U^K\|_{\infty}+\frac{y^2}{2}\right) +(\beta_t-\alpha)^3\|\na U^K\|_{\infty} + \frac{y^2}{\beta_t-\alpha} -\gamma_t\left(\beta_t-\alpha\right)y^2 \\& \qquad+ \beta_t^{-1}\gamma_t\left(d\left(\beta_t-\alpha\right) + \left(\beta_t-\alpha\right)^2y^2\right) \bigg)\phi_t \\ &\leqslant \bigg( \beta'_0\|U^K\|_{\infty} +(\beta_T-\alpha)^3\|\na U^K\|_{\infty}+ Ld\left(\beta_T-\alpha\right)  + \left( \frac{\beta'_t}{2} + \frac{1}{\beta_t-\alpha} - \frac{\gamma_t\alpha\left(\beta_t-\alpha\right)}{\beta_t} \right)y^2 \bigg)\phi_t.
\end{align*}
We can choose $\beta_0$ great enough so that $ \frac{\beta'_t}{2} +\frac{1}{\beta_t-\alpha} - \frac{\gamma_t\alpha(\beta_t-\alpha)}{\beta_t}\leqslant 0$ for all $t\geqslant 0$. We then classically get for $t\in\left[0,T\right]$ :
\[ \E\left(\phi_{t\wedge\tau_N}\left(Z^K_{t\wedge\tau_N}\right)\right) \leqslant e^{C_Tt}\E\left(\phi_0\left(X_0,Y_0\right)\right)\]
where $C_T=\beta_0'\|U^K\|_{\infty}+(\beta_T-\alpha)^3\|\na U^K\|_{\infty} + Ld(\beta_T-\alpha)$. The fact that $f_0^K$ has a compact support and Fatou's lemma then yield the result.
\end{proof}


\begin{lemma}\label{regu}
Fix $K,c>0$. Then there exists $b_K$ such that if $\beta_0\geqslant b_K$, the law $f^K_t$ of the process defined by the Equation~\eqref{eq1}, with an initial condition $f_0^K\in\mathcal{C}^{\infty}$ with compact support, is smooth, \lj{bounded along with its derivative,} and satisfies :
\[ \exists \alpha>0,\forall t\geqslant0, \exists C_t \text{ such that } f_t^K + |\na f_t^K| \leqslant C_te^{-\alpha |y^2|} \]
and $t\mapsto C_t$ is locally bounded.
\end{lemma}

\begin{proof}
\pierre{First, by Ito's formula, $f_t^K(x,y)$ is a weak measure solution of the forward equation
\begin{equation}\label{eq:kolmo-forward}
\begin{array}{lll}
\partial_t f_t^K &=& - y \cdot  \na_x f_t^K + \na U^K(x)\cdot  \na_y f_t^K + \gamma_t \na_v\cdot \po v f_t^K\pf + \gamma_t \beta_t^{-1}\Delta_v f_t^K \\
& = & \tilde L_t f_t^K + d \gamma_t  f_t^K\,,
\end{array}    
\end{equation}
where $\tilde L_t$ is the generator of the process $(\tilde X,\tilde Y)$ solving
\[\dd \tilde X_t = - \tilde Y_t\dd t \,,\qquad \dd \tilde Y_t = \na U^K(\tilde X_t)\dd t + \gamma_t \tilde Y_t \dd t + \sqrt{2\gamma_t\beta_t^{-1}} \dd B_t\,.\]
Using again Ito's formula, we get that the function given by 
\[(x,y,t) \mapsto  e^{d\int_0^t\gamma_s \dd s }\mathbb E_{x,y} \po f_0^K \po \tilde X_t,\tilde Y_t\pf \pf\,,\]
also solves equation~\eqref{eq:kolmo-forward}. Uniqueness of the weak solution of \eqref{eq:kolmo-forward} is ensured by \cite[Theorem 9.8.7]{Krylov}, hence
\[f_t^K(x,y) = e^{d\int_0^t\gamma_s \dd s }\mathbb E_{x,y} \po f_0^K \po \tilde X_t,\tilde Y_t\pf \pf\,,\]
from which we immediately get that $f_t^K$ is bounded uniformly over $[0,T]$ for all $T>0$. Thanks to \cite[Theorem 1]{Locherbach}, $f_t^K$ is smooth for all $t\geqslant 0$. We can always differentiate  \eqref{eq:kolmo-forward} in a weak sense (i.e. once integrated with respect to a smooth compactly supported function of time and space and formally integrating by parts), from which we get that $\na f_t^K$ is a weak solution of
\[\partial_t \na f_t^K = \tilde L_t \na f_t^K + J_t \na f_t^K\,,\]
where $\tilde L_t$ acts component-wise on $\na f_t^K$ and
\[J_t(x,y) = \begin{pmatrix} d \gamma_t I_d  & \na^2 U^K(x) \\ - I_d & 2 d \gamma_t I_d\end{pmatrix} \,,\]
which is bounded uniformly over $[0,T]$ for all $T>0$. For a fixed $T>0$, considering $(W_t)_{t\geqslant 0}$ the matrix-valued process solution of $\dd W_t = W_t J_{T-t}( \tilde X_t,\tilde Y_t)$ with $W_0=I_d$ and using again the uniqueness of the weak solution of the PDE, we get the Feynman-Kac representation
\[\na f_t^K(x,y) = \mathbb E_{x,y} \po W_t \na f_0^K(\tilde X_t,\tilde Y_t) \pf \,,\]
which can be obtained by Itô's formula for the process $(\tilde X_t,\tilde Y_t,W_t)_{t\geqslant 0}$ applied to the test function $h(x,y,w,t) = w\na f_t^K(x,y)$. Hence, $\na f_t^K$ is uniformly bounded over $[0,T]$, and the same argument applies to all derivatives of $f_t^K.$

Now, let us prove the Gaussian bound of the statement, starting with $f_t^K$.
}
For $z\in\R^{2d}$, let $\epsilon(z)=\min(1,\frac{f^K_t(z)}{2\|\na f^K_t\|_{\infty}})$. If $z'\in B(z,\epsilon(z))$, $f_t^K(z')\geqslant f_t^K(z) - \|\na f_t^K\|_{\infty}\epsilon(z)\geqslant \frac{f_t^K(z)}{2}$. Then for any $z=(x_1,y_1)\in\R^{2d}$ :
\begin{align*}
\E(e^{H_K(Z^K_t)/2}) &=\int_{\R^d\times\R^d} e^{H_K(x,y)/2} f^K_t(x,y)\dd x\dd y \\ &\geqslant \int_{B\left(z,\epsilon(z)\right)} e^{H_K(x,y)/2} f^K_t(x,y)\dd x\dd y \\ &\geqslant e^{-\frac{\|U\|_{\infty}}{2}}e^{\frac{|y_1|^2-1}{4}}\frac{f_t^K(z)}{2}\omega_{2d}\epsilon(z)^d
\end{align*}
where $\omega_{2d}$ is the volume of the unit-sphere in dimension 2d. Then there are two possibilities. If $f^K_t(x,y)\geqslant 2\|\na f^K_t\|_{\infty}$, then 
\[ f_t^K(x,y) \leqslant \left(2e^{\frac{1}{4}}e^{\frac{\|U\|_{\infty}}{2}}\frac{\E(e^{H_K(Z^K_t)/2})}{\omega_{2d}}\right)e^{-\frac{y^2}{4}}. \]
Otherwise, $f^K_t(x,y)\leqslant 2\|\na f^K_t\|_{\infty}$ and 
\[ f^K_t(x,y) \leqslant  2\|\na f_t^K\|_{\infty}^{\frac{d}{d+1}}\left( e^{\frac{1}{4}}e^{\frac{\|U\|_{\infty}}{2}}\frac{\E(e^{H_K(Z^K_t)/2})}{\omega_{2d}}\right)^{\frac{1}{d+1}} e^{-\frac{y^2}{4(d+1)}}.\]
In any case, using the fact that $t\mapsto\E\left(e^{H_K\left(Z^K_t\right)/2}\right)$ is locally bounded for $\beta_0$ great enough, we have the result for $f_t^K$ with $\alpha=\frac{1}{4(d+1)}$.
Recall the definition of the approximation functions $\eta_m$ from Proposition~\ref{approx1}. We use them here with $U$ replaced by $U^K$ (or equivalently with $H$ replaced by $H_K$), but it doesn't change any of their properties since now $x$ is in a compact set. We now turn to $\na f_t^K$ by using the approximation functions $\eta_m$ from Section~\ref{EnergyBounds} and integrating by parts:
\begin{align*}
\int_{\R^d\times\R^d}& \eta_m e^{\frac{H_K}{4}} \left|\na f_t^K\right|^2\dd z = -\int_{\R^d\times\R^d}\na\left( \eta_m e^{\frac{H_K}{4}}\na f^K_t\right) f_t^K \dd z \\ &= -\int_{\R^d\times\R^d} \left( \na_y\eta_m\cdot \na f^K_te^{\frac{H_K}{4}} + \frac{\eta_m}{2}e^{\frac{H_K}{4}}\na H_K\cdot \na f^K_t + \eta_me^{\frac{H_K}{4}}\Delta f^K_t\right)f_t^K \dd z \\&\leqslant C\int_{\R^d\times\R^d}e^{\frac{H_K}{2}}f_t^K \dd z = C\E\left(e^{H_K\left(Z^K_t\right)/2}\right)
\end{align*}
for some constant $C>0$ independent from $m$, where we used uniform bounds on the two first derivatives of $f_t^K$. We can then conclude in the same way as for $f_t^K$.
\end{proof}

Setting $\phi_{\beta,\gamma}(h) = |(\na_x+\na_y)h|^2 + \sigma(\beta,\gamma)h^2$ with $\sigma(\beta,\gamma)=\frac{1}{2} + 2\sqrt{\gamma^{-1}\beta}(1+\|\na U^K\|_{\infty} + \gamma)^2$, we consider the modified $H^1$-norm (and its truncated version for $m\geqslant 1$):
\begin{align*}
    N(t,\beta,\gamma) &= \int_{M_K\times\R^d} \phi_{\beta,\gamma}\left(\frac{f_t^K}{\mu^K_{\beta}}\right)\dd\mu^K_{\beta}\\
    N_m(t,\beta,\gamma)& = \int_{M_K\times\R^d} \eta_m\phi_{\beta,\gamma}\left(\frac{f_t^K}{\mu^K_{\beta}}\right)\dd\mu^K_{\beta}  
\end{align*}
as well as 
\begin{align*}
   I(t,\beta) &= \int_{M_K\times\R^d} \left|\na \frac{f_t^K}{\mu^K_{\beta}}\right|^2 \dd\mu_{\beta}^K \\
   I_m(t,\beta) &= \int_{M_K\times\R^d} \eta_m \left|\na \frac{f_t^K}{\mu^K_{\beta}}\right|^2 \dd\mu_{\beta}^K
\end{align*}
 To study those functionals along the dynamic, write $\tilde{H}_m(t)=H_m(t,\beta_t,\gamma_t)$, $\tilde{H}(t)=H(t,\beta_t,\gamma_t)$, $\tilde{I}_m(t)=I_m(t,\beta_t)$ and $\tilde{I}(t)=I(t,\beta_t)$.

The main technical tool that we will need in order to study the evolution of those functionals will be, given $\phi:\mathcal{C}^{\infty}\mapsto\mathcal{C}^{\infty}$, quantities of the form :
\begin{equation}\label{eq:defGamma}
    \Gamma_{\phi,L}(h) = L\left(\phi(h)\right) - D_h\left(\phi\right)(h)L(h),
\end{equation}
where $D_h(\phi)$ is the pointwide differential of $\phi$, see \cite{Pierro2}.
The reason is that  for regular enough $h$, by writing $L_{\beta,\gamma}^K$ the generator~\eqref{gen} on $M_K$ with fixed $\beta$ and $\gamma$, we have 
\[ \partial_t\int_{M_K\times\R^d}\phi\left( e^{tL_{\beta,\gamma}^K}h \right) \dd \mu_{\beta}^K = -\int_{M_K\times\R^d} \Gamma_{\phi,L_{\beta,\gamma}^K}\left( e^{tL_{\beta,\gamma}^K}h \right)\dd \mu_{\beta}^K.\]
This kind of quantities has been studied in \cite{Pierro2}, where the author showed among other things :
\lj{\begin{lemma}\label{Gam}
Let $\phi(h) = |(\na_x+\na_y)h|^2 + \sigma h^2$ where $\sigma = \frac{1}{2} + \sqrt{\gamma\beta^{-1}}(1 + \|\na U\|_{\infty} +\gamma)^2$. Then, for all $\beta>0$, $\gamma>0$, $h\in\mathcal{C}^{\infty}$:
\[ \Gamma_{\phi,L_{\beta,\gamma}^{K,\lj{*}}}(h) \geqslant \frac{1}{2} |\na h|^2 + \Gamma((\na_x+\na_y)h),\] where $\Gamma((\na_x+\na_y)h)=\sum_{i=0}^d\Gamma((\partial_{x_i}+\partial_{y_i})h)=\sum_{i=0}^d\gamma_t\beta^{-1}_t|\na_y(\partial_{x_i}+\partial_{y_i})h|^2$.
\end{lemma}

\begin{proof}
This is \cite[Example 3]{Pierro2}. Let us simply recall the key idea (to alleviate notations we omit the generator in the subscript of $\Gamma$). First, $\Gamma$ is linear in $\phi$, hence for $\phi$ as defined above:
\[\Gamma_{\phi} = \sigma\Gamma + \Gamma_{|(\na_x+\na_y)\cdot|^2},\]where $\Gamma$ is the classical "carr\'e du champs" $\Gamma(h)=\gamma_t\beta^{-1}_t|\na_y h|^2$. We then have the following equality:
\[\Gamma_{|(\na_x+\na_y)\cdot|^2}(h) = \Gamma((\na_x+\na_y)h) + (\na_x+\na_y)h\cdot[L,\na_x+\na_y]h,\]where the brackets denotes the commutators of two operators. An elementary computation concludes.
\end{proof}}

\begin{proof}[Proof of Proposition~\ref{evol0}]
First, $\beta_0$ must be great enough so that we can apply all previous lemmas of Section~\ref{hypo1}. Let's first justify why $\tilde{N}_m$, $\tilde{N}$, $\tilde{I}_m$ and $\tilde{I}$ are finite. They are all finite at time $0$ because $f_0$ has compact support. For $m\in\N$, $\tilde{N}_m$ and $\tilde{I}_m$ are finite for all times as integrals of continuous functions with compact support. 

 Write $\nabla^*_t= -\nabla+\beta_t\na H_K$ the dual of the gradient operator $\nabla$ in $L^2\left(\mu_{\beta_t}^K\right)$.
 Integrating by parts,
\begin{align*}
\tilde{I}_m &= \int_{M_K\times\R^d} \eta_m \left|\na \frac{f_t^K}{\mu^K_{\beta_t}}\right|^2 \dd\mu^K_{\beta_t}\\ &= \int_{M_K\times\R^d} \na_t^*\left(\na \eta_m \na \frac{f_t^K}{\mu^K_{\beta_t}} \right) f_t^K \dd z\\& \leqslant \int_{M_K\times\R^d} C_te^{(\beta_t-\alpha)H_K(z)} f_t^K(z)\dd z = C_t\E\left(e^{(\beta_t-\alpha)H_K(Z^K_t)}\right)
\end{align*} 
for some $\alpha >0$ and $C_t>0$ locally bounded and independent from $m$, using Lemma~\ref{regu} and that the derivatives of $f_t^K$ are bounded. Lemma~\ref{espexp} and the monotonous convergence of $\tilde{I}_m$ towards $\tilde{I}$ then implies that $\tilde{I}$ is locally bounded.
We conclude that $\tilde{N}$ is also locally bounded thanks to the Poincar{\'e} inequality of Proposition~\ref{poin}.
As in \eqref{eq:L*}, we have that the dual in $L^2(\mu_\beta^K)$  of the generator $L_t$ is 
\[L^{\ast}_{t} = -y\cdot\nabla_x + (-\gamma_t y + \nabla_xU^K)\cdot \nabla_y + \gamma_t\beta_t^{-1}\Delta_y\]and that $f_t^K$ solves $\partial_tf_t^K=L_t^{\ast}(f_t^Ke^{\beta_t H_K})e^{-\beta_t H_K}$.
With the regularity result from Lemma~\ref{regu}, and the compactly supported approximation functions $\eta_m$, one can differentiate $N_m$ and $I_m$  to get for all $m$:
\begin{align*}
(\partial_tN_m)(t,\beta_t,\gamma_t) &= \int_{M_K\times\R^d} \eta_mD_h\phi_{\beta_t,\gamma_t}\left(h_t^K\right)\frac{\partial_tf^K_t}{\mu^K_{\beta_t}}\dd \mu^K_{\beta_t} \\ &= \int_{M_K\times\R^d} \eta_mD_h\phi_{\beta_t,\gamma_t}\left(h_t^K\right)L^{\ast}_t\left(h_t^K\right)\dd \mu_{\beta_t}^K - \int_{M_K\times\R^d} L_t^{\ast}\left(\eta_m\phi_{\beta_t,\gamma_t}\left(h_t^K\right)\right)\dd\mu^K_{\beta_t}  \\ &= -\int_{M_K\times\R^d} \eta_m\Gamma_{\phi_{\beta_t,\gamma_t}}\left(h_t^K\right)\dd\mu_{\beta_t}^K - \int_{M_K\times\R^d} L_t^{\ast}(\eta_m)\phi_{\beta_t,\gamma_t}\left(h_t^K\right)\dd\mu_{\beta_t}^K \\  &\ \ \ \ \ - 2\int_{M_K\times\R^d} \Gamma_t\left(\eta_m,\phi_{\beta_t,\gamma_t}\left(h_t^K\right)\right)\dd\mu_{\beta_t}^K.
\end{align*}
From Lemma~\ref{Gam}, we get $\Gamma_{\phi_{\beta,\gamma}}(h) \geqslant \frac{1}{2}|\na h|^2 + \Gamma_t((\na_x+\na_y)h)$, and since 
\begin{multline*}
\int_{M_K\times\R^d} \Gamma_t\left(\eta_m,\phi_{\beta_t,\gamma_t}\left(h_t^K\right)\right)\dd\mu_{\beta_t}^K \\= -\frac{1}{2}\int_{M_K\times\R^d} \left(\eta_mL^{\ast}_t\phi_{\beta_t,\gamma_t}\left(h_t^K\right) + L_t^{\ast}\eta_m\phi_{\beta_t,\gamma_t}\left(h_t^K\right) \right)\dd\mu_{\beta_t}^K \\= -\frac{1}{2} \int_{M_K\times\R^d} (L_t+L_t^{\ast})\left(\eta_m\right)\phi_{\beta_t,\gamma_t}\left(h_t^K\right)\dd\mu_{\beta_t}^K
\end{multline*}
and $L_t\eta_m\leqslant \frac{C\beta_t}{m}$, we get :
\begin{align*}
(\partial_tN_m)(t,\beta_t,\gamma_t) &\leqslant -\frac{1}{2} \int_{M_K\times\R^d} \eta_m \left|\na h_t^K\right|^2 \dd\mu_{\beta_t}^K - \int_{M_K\times\R^d} \Gamma_t\left((\na_x+\na_y)h_t^K\right)\dd\mu_{\beta_t}^K \\ & \ \ \ \ \ + \int_{M_K\times\R^d} L_t(\eta_m)\phi_{\beta_t,\gamma_t}\left(h_t^K\right) \dd\mu_{\beta_t}^K  \\&\leqslant -\frac{1}{2} \int_{M_K\times\R^d} \eta_m \left|\na h_t^K\right|^2 \dd\mu_{\beta_t}^K - \int_{M_K\times\R^d} \Gamma_t\left((\na_x+\na_y)h_t^K\right)\dd\mu_{\beta_t}^K \\& \ \ \ \ \ + \frac{C\lj{\beta_t}}{m}\int_{M_K\times\R^d} \phi_{\beta_t,\gamma_t}\left(h_t^K\right) \dd\mu_{\beta_t}^K.
\end{align*}

Now we look at the derivative in $\beta$, knowing that $\partial_{\beta}Z \leqslant 0$ by writting : 
\[N_m(t,\beta,\gamma) = \int_{M_K\times\R^d} \left(\eta_m\left|(\na_x+\na_x)\ln\left(\frac{f_t^K}{\mu^K_{\beta}}\right)\right|^2\frac{f_t^K}{\mu^K_{\beta}}+ \eta_m\sigma(\beta,\gamma)\frac{f_t^K}{\mu^K_{\beta}}\right)\dd f_t^K.\]
This gives
\begin{multline*}
(\partial_{\beta}N_m)(t,\beta_t,\gamma_t) \leqslant \partial_{\beta}\sigma(\beta_t,\gamma_t)\int_{M_K\times\R^d} \eta_m\left(h_t^K\right)^2\dd\mu_{\beta_t}^K + \sigma(\beta_t)\int_{M_K\times\R^d} \eta_m (U^K(x) + \frac{y^2}{2})\left(h_t^K\right)^2\dd\mu_{\beta_t}^K \\+ \int_{M_K\times\R^d} \eta_m(U^K(x) + \frac{y^2}{2})\left|(\na_x + \na_y)h_t^K\right|^2\dd\mu_{\beta_t}^K\\ + \int_{M_K\times\R^d} 2\eta_m(\na_x+\na_y)\ln\left(h_t^K\right)\cdot (\na_x+\na_y)\partial_\beta(-\ln(\mu^K_{\beta_t})) h_t^K\dd f_t^K.
\end{multline*}
We now have to treat all those terms. We will use that $N_m\leqslant N$. First, since $U^K$ is bounded,
\[\partial_{\beta}\sigma(\beta_t,\gamma_t)\int_{M_K\times\R^d} \eta_m\left(h_t^K\right)^2\dd\mu_{\beta_t}^K + \sigma(\beta_t)\int_{M_K\times\R^d} \eta_m U^K(x)\left(h_t^K\right)^2\dd\mu_{\beta_t}^K \leqslant C(1+\beta_t^{2})\tilde{N}\]
for some $C>0$ (in the rest of the proof we denote by $C$ several constants which do not depend on $t$ nor $m$). 
From Proposition~\ref{Vila}, applied with $g=h_t^K$, we get : 
\[\sigma(\beta_t,\gamma_t)\int_{M_K\times\R^d}\eta_m\frac{y^2}{2}\left(h_t^K\right)^2\dd\mu_{\beta_t}^K \leqslant C(1+\beta_t^{2})\tilde{N}.\]
Again, because $U^K$ is bounded : \[\int_{M_K\times\R^d} U(x)\left|(\na_x+\na_y)h_t^K\right|^2\dd\mu_{\beta}^K \leqslant C\tilde{N}.\]
Using again Proposition \ref{Vila} we get : 
\[\int_{M_K\times\R^d} \eta_m \frac{y^2}{2}\left|(\na_x+\na_y)h_t^K\right|^2\dd \mu^K_{\beta_t} \leqslant C\left(\tilde{N} + \sqrt{\gamma_t^{-1}\beta_t}\int_{M_K\times \R^d}\Gamma_t\left((\na_x+\na_y)h_t^K\right)\dd\mu_{\beta_t}^K\right).\]
Finally, using that $a.b\leqslant \frac{a^2}{2} + \frac{b^2}{2}$ and $-\partial_{\beta}\ln(\mu_{\beta}^K) = U^K(x) + \frac{y^2}{2} + \frac{\partial_{\beta}Z_{\beta}}{Z_{\beta}}$, we get :
\begin{align*}
\int_{M_K\times\R^d}& \eta_m(\na_x+\na_y)\ln\left(\frac{f_t^K}{\mu^K_{\beta}}\right)\cdot (\na_x+\na_y)\partial_\beta(-\ln(\mu^K_{\beta_t})) \frac{f_t^K}{\mu^K_{\beta_t}}\dd f_t^K \\ &\leqslant C\int_{M_K\times\R^d} \eta_m\left|(\na_x+\na_y)h_t^K\right|^2\left(\frac{f_t^K}{\mu^K_{\beta_t}}\right)^2\dd\mu_{\beta_t}^K \\&\ \ \ \ \ + C\int_{M_K\times\R^d} \eta_m|(\na_x+\na_y)(U(x)+\frac{y^2}{2})|^2\left(h_t^K\right)^2\dd\mu_{\beta_t}^K \\
&\leqslant C\po N_m + \|\na U\|^2_{\infty}\int\eta_m\left(h_t^K\right)^2\dd\mu_{\beta_t}^K + \int y^2 \left(h_t^K\right)^2\dd\mu_{\beta_t}^K\pf 
\end{align*}
and we conclude as the second term. Finally, since $\gamma$ only appears in the definition of $\sigma$, we have :
\[ (\partial_{\gamma}N_m)(t,\beta_t,\gamma_t) =\partial_\gamma\sigma\int_{M_K\times\R^d}\eta_m \left(h_t^K\right)^2\dd \mu^K_{\beta_t} \leqslant C\gamma_t\tilde{N}\]
From the previous computation, and from the fact that $\gamma'_t\leqslant L\beta'_t$, we get :
\begin{align*}
\tilde{N}_m'(t) &= (\partial_tN_m)(t,\beta_t,\gamma_t) + \beta'_t(\partial_{\beta}N_m)(t,\beta_t,\gamma_t) + \gamma'_t(\partial_{\gamma}N_m)(t,\beta_t,\gamma_t) \\ &\leqslant -\frac{1}{2}\tilde{I}_m(t) - \int_{M_K\times\R^d}\eta_m \Gamma_t\left(Ah_t^K\right)\dd\mu_{\beta_t}^K\\ &\qquad +C\beta'_t\beta_t \int_{M_K\times\R^d} \Gamma_t\left(Ah_t^K\right)\dd\mu_{\beta_t}^K + \frac{C\beta_t}{m}\tilde{N}_t + C\beta'_t(1+\beta_t^{2})\tilde{N}_t.
\end{align*}
By integration, we get for all $0\leqslant s\leqslant t$:
\begin{align*}
\tilde{N}_m(t)-\tilde{N}_m(s) &\leqslant \int_s^t \Bigg( -\frac{1}{2}\tilde{I}_m(u)+ \frac{C\beta_u}{m}\tilde{N}_u + C\beta'_u(1+\beta_u^{2})\tilde{N}_u \\& \qquad - \int_{M_K\times\R^d}\eta_m \Gamma_t\left(Ah_t^K\right)\dd\mu_{\beta_u}^K +C\beta'_u\beta_u \int_{M_K\times\R^d} \Gamma_t\left(Ah_t^K\right)\dd\mu_{\beta_u}^K\Bigg)\dd u  .
\end{align*}
We consider $\beta_0$ great enough so that $1-C\beta'_t\beta_t\geqslant0$ for all $t$. Using monotonous convergence, the fact that $\tilde{I}\geqslant \tilde{I}_m$, and Fatou's lemma we get :
\[ \tilde{N}(t)-\tilde{N}(s) \leqslant \int_s^t \left( -\frac{1}{2}\tilde{I}(u) + C\beta'_u(1+\beta_u^{2})\tilde{N}_u\right)\dd u. \]
Now, from the Poincar{\'e} inequality  of Proposition~\ref{poin} and the definition of $\sigma$,  
\[\tilde{\lambda}(\beta)\tilde{N}(t)\leqslant\tilde{I}(t)\] for some $\tilde \lambda$ satisfying $\frac{1}{\beta}\ln(\tilde \lambda)\rightarrow-c^*$, so that we can find a $\lambda_0$ such that \[\tilde \lambda(\beta_t)\geqslant \lambda_0e^{-\beta_t\frac{(c+c^{\ast})}{2}} = \frac{\lambda_0}{(e^{c\beta_0}+t)^{1-\alpha}},\] where $\alpha=(c-c^{\ast})/(2c)$.
Taking $\beta_0$ large enough so that $ C\beta'_t(1+\beta_t^{2})\leqslant \lambda_0/4(e^{c\beta_0}+t)^{-1+\alpha}$, we have finally obtained that, if $\beta_0\geqslant \tilde{b}_K$ for some $\tilde b_K$, then for all $ 0\leqslant s\leqslant t$ 
\[ \tilde{N}(t)-\tilde{N}(s) \leqslant \int_s^t \left(-\frac{\lambda_0}{4(e^{c\beta_0}+u)^{1-\alpha}}\tilde{N}(u)\right)\dd u\] and this concludes (as explained at the beginning of this section).
\end{proof}

\section{ \texorpdfstring{$H^k$}{Hk}-hypocoercivity}\label{hypo2}

In all this section, whose goal is to prove Proposition~\ref{evol1}, we use the definitions and notations of Section~\ref{fullprocess}. In particular, $f_t^K$ stands for the law of the process defined by Equation~\eqref{eq2}, $\mu^K_{\beta}= e^{-\beta H_K(z)}\dd z/\mathcal{Z}_{\beta}^K$ where $\mathcal{Z}_{\beta}^K$ makes $\mu^K_{\beta}$ a probability density on $M_K$, and  $h_t^K = f_t^K/\mu^K_{\beta_t}$. Similarly to the previous section, according to \cite{Taniguchi}, $f_t^K$ is a smooth function and solves \[\partial_tf_t^K=L^{K,*}_t\left(\frac{f_t^K}{\mu_{\beta_t}^K}\right)\mu_{\beta_t}^K,\] but here the dual in $L^2(\mu_{\beta_t})$ of the generator of the process is
\begin{align*}
	L^{K,*}_t = -\beta_t^{-1}\sigma(y)\na_x^*\na_x - \gamma_t\beta_t^{-1}\na_y^*\na_y + \beta_t^{-1}\left( \na_x^*\na_y - \na_y^*\na_x \right),
\end{align*} 
\lj{where the dual are taken in $L^2(\mu^K_{\beta})$:
\begin{align*}
&\na_x^* = -\na_x\cdot  + \beta \na U^K \cdot \\ &\na_y^* = -\na_y\cdot + \beta \na W\cdot.
\end{align*}}
Indeed, the additional diffusive part in $x$ has been designed to be reversible. In order to prove Proposition~\ref{evol1}, we introduce for $m\in\N$ the classical $H^m$-Sobolev norms on $(M_K^2,\mu_\beta^K)$ given by
\[\|h\|_{H^m(\mu_{\beta}^K)}=\sum_{\alpha\in\mathbb{N}^{2d},|\alpha|\leqslant m}\int_{M_K^2}|\partial^{\alpha}h|^2 \dd \mu_\beta^K.\]
The general strategy is similar to the $L^2$ case of Section~\ref{hypo1}, namely we will prove that $\|h^K_t-1\|_{H^m}$ goes to zero for all $m$, and conclude by a Sobolev embedding for $m$ large enough. The constant in the Sobolev embedding depends on the time $t$ which should be compensated by the fact  $\|h^K_t-1\|_{H^m}$   goes to zero fast enough.  As in $L^2$ case, we need to introduce some modified Sobolev norm to deal with the lack of dissipativity in the $x$ variable in some part of the space. Following \cite{Hyp}, for $m\in\mathbb{N}$, we consider a modified $H^m$-Sobolev norm of the form 
\begin{multline}\label{defN}
N_m(t,\beta) =\\ \int_{M_K\times M_K} \left( \frac{f_t^K}{\mu_{\beta}^K} -1 \right)^2 + \sum_{k=1}^m \left(\sum_{i=0}^k \omega_{i,k}(\beta)\left|\na_x^i\na_y^{k-i}\frac{f_t^K}{\mu_{\beta}^K}\right|^2 + \omega_k(\beta)\na_x^{k-1}\na_y\frac{f_t^K}{\mu_{\beta}^K}\cdot \na_x^m \frac{f_t^K}{\mu_{\beta}^K}\right) \dd\mu^K_{\beta}
\end{multline} 
for some weights $\omega$ to be fixed later on. 
\lj{Recall the definition
\[ \mathcal{M} = \left\{ y\in M_K;\na^2 W(y) = I_d \right\}. \]
If $y\in\mathcal M$, then we obtain $\na_x$ in the derivative of 
\[
\int_{M_K\times M_K} \na_x^{k-1}\na_y\frac{f_t^K}{\mu_{\beta}^K}\cdot \na_x^m \frac{f_t^K}{\mu_{\beta}^K} \dd \mu^K_\beta
\]
using the commutator $[\na_y,\na W\cdot \na_x]$, and for $y\notin\mathcal M$, it comes from the $\sigma(y)\na_x^*\na_x$ part of $L^{K,*}_t$.}
Here we used the notations
\begin{align*}
|\na_x^i\na_y^{k-i}h|^2&=\sum_{|\alpha_1|= i;|\alpha_2|= k-i} |\partial_x^{\alpha_1}\partial_y^{\alpha_2} h|^2\\
\na_x^{k-1}\na_yh\cdot \na_x^mh&=\sum_{|\alpha|=k-1}\sum_{j=1}^{d}\partial_x^{\alpha}\partial_{y_j}h\partial_x^{\alpha}\partial_{x_j}h.
\end{align*}
We also define $\tilde{N}_m(t)=N_m(t,\beta_t)$. Since the process is now in a compact set, we do not need the approximation functions $\eta_m$, and the subscript here indicates the order of the Sobolev norm in contrast to the previous section. Besides, in Section~\ref{hypo1}, we had to keep track of the dependency of the constant in $\beta_0$, as some uniformity in time was necessary in Proposition~\ref{boundedness} for the renewal argument of the proof of Proposition~\ref{prop:born{\'e}}. This is no longer the case here. 

In order to study the evolution of $N_m$, we need first to have some commutation results and control over the derivative of the $L^2$-norm of $\na_x^{m_1}\na_y^{m_2}h$ as in Lemma~\ref{Gam}. Here, $\na_x^{m_1}\na_y^{m_2}h$ denotes the vector of all derivative of $h$ of order $m_1$ on $x$ and $m_2$ on $y$:
\[\na_x^{m_1}\na_y^{m_2}h=\left\{\partial^{\alpha_1}_x\partial^{\alpha_2}_yh\;\middle|\; |\alpha_1|=m_1,\, |\alpha_2|=m_2\right\}\]
The method used here is adapted from \cite{Hyp} to take into account the time-inhomogeneity and the new dynamic. Recall the notation \eqref{eq:defGamma} for generalized $\Gamma$ operators.

\begin{lemma}\label{commu}
Let $m_1,m_2\in\mathbb{N}$, then there exists some constant $\theta>0$ such that for all $h\in\mathcal{C}^{\infty}\left(M_K^2\right)$:
\begin{multline*}
\Gamma_{t,|\na_x^{m_1}\na_y^{m_2}\cdot|^2}(h) \geqslant \beta_t^{-1}\left( \gamma_t|\na_x^{m_1}\na_y^{m_2+1}h|^2 +\sigma |\na_x^{m_1+1}\na_y^{m_2}h|^2\right) \\  -\theta(1+\gamma_t) \sum_{\underset{k+l\leqslant m_1+m_2}{k=1,l=1}}^{m_1+1,m_2+1}|\na_x^{k}\na_y^{l}h|^2 - \theta\mathbbm{1}_{m_2\geqslant 1}\sum_{l=1}^{m_2}|\na^l\sigma|^2|\na_x^{m_1+2}\na_y^{m_2-l}h|^2.
\end{multline*}
\end{lemma}

\begin{proof}
As in \cite[Lemma 10]{Pierro}, we have for smooth $h$ :
\[ \Gamma_{t,|\na_x^{m_1}\na_y^{m_2}\cdot|^2}(h) = \Gamma_t(\na_x^{m_1}\na_y^{m_2}h) + \na_x^{m_1}\na_y^{m_2}h\cdot \left[L_t^{K,*},\na_x^{m_1}\na_y^{m_2}\right]h.\]
Then first :
\[ \Gamma_t(\na_x^{m_1}\na_y^{m_2}h) = \beta^{-1}_t\left( \gamma_t|\na_x^{m_1}\na_y^{m_2+1}h|^2 +\sigma |\na_x^{m_1+1}\na_y^{m_2}h|^2\right).\]
Then, we can write $L^{K,*}_t = v\cdot \na + \beta_t^{-1}\sigma\Delta_x + \beta_t^{-1}\gamma_t \Delta_y$ with the drift 
\[v(x,y)=\begin{pmatrix} -\na W(y) - \sigma(y)\na U^K(x) \\ \na U^K(x) - \gamma_t\na W(y) \end{pmatrix}.\]
First,
\[  \left[\beta_t^{-1}\gamma_t \Delta_y,\na_x^{m_1}\na_y^{m_2}\right] = 0.\]
Then, using Cauchy-Schwarz inequality,
\begin{align*}
\na_x^{m_1}\na_y^{m_2}h\cdot &\left[\beta_t^{-1}\sigma \Delta_x,\na_x^{m_1}\na_y^{m_2}\right]h = \na_x^{m_1}\na_y^{m_2}h\cdot \beta_t^{-1}\sum_{l=1}^{m_2}\binom{m_2}{l}\na_y^l\sigma\cdot \na_y^{m_2-l}\na_x^{m_1} \Delta_x h \\ &\geqslant - \beta^{-1}_t\left(|\na_x^{m_1}\na_y^{m_2}h|^2 + \theta \sum_{l=1}^{m_2}|\na^l\sigma|^2|\na_x^{m_1+2}\na_y^{m_2-l}h|^2\right).
\end{align*}
Notice that if $m_2=0$, then $\na_x^{m}h\cdot \left[\beta_t^{-1}\sigma \Delta_x,\na_x^{m}\right]h =0$. The next term is:
\begin{align*}
\left[\left(-\na W-\sigma\na U^K\right)\na_x,\na_x^{m_1}\na_y^{m_2}\right]h &= \sum_{l=1}^{m_2}\binom{m_2}{l} \na_y^{l+1} W \na_x^{m_1+1}\na_y^{m_2-l}h \\&+ \sum_{\underset{\underset{(k,l)\neq (0,0)}{l\leqslant m_2}}{k\leqslant m_1}} \binom{m_1}{k}\binom{m_2}{l}\na_y^l\sigma\na_x^{k+1}U\na_x^{m_1-k+1}\na_y^{m_2-l}h 
\end{align*}
and 
\begin{align*}
\left[\left(-\gamma_t\na W+\na U^K\right)\na_y,\na_x^{m_1}\na_y^{m_2}\right]h &= \gamma_t\sum_{l=1}^{m_2}\binom{m_2}{l} \na_y^{l+1} W \na_x^{m_1}\na_y^{m_2-l+1}h \\&- \sum_{k=1}^{m_1} \binom{m_1}{k}\na_x^{k+1}U^K\na_x^{m_1-k}\na_y^{m_2+1}h.
\end{align*}
Putting those last two lines together and using Cauchy-Schwarz we get :
\begin{multline*}
\na_x^{m_1}\na_y^{m_2}h.\left[v\cdot\na,\na_x^{m_1}\na_y^{m_2}\right]h \\ \geqslant -\theta(1+\gamma_t)\left( \sum_{l=1}^{m_2} |\na_x^{m_1+1}\na_y^{m_2-l}h|^2 + \sum_{\underset{1\leqslant l\leqslant m_2}{1\leqslant k\leqslant m_1}}|\na_x^{m_1-k+1}\na_y^{m_2-l}h|^2 \right. \\ \left. + \sum_{l=1}^{m_2} |\na_x^{m_1}\na_y^{m_2-l+1}h|^2 + \sum_{k=1}^{m_1} |\na_x^{m_1-k}\na_y^{m_2+1}h|^2\right) .
\end{multline*}
This concludes the proof.
\end{proof}

\begin{lemma}\label{lemPm}
For all $m\in\mathbb{N}$  there exists some constant $\theta>0$ such that for all smooth $h$, denoting $P^m(t,\beta) = \int_{M_K\times M_K} \na_x^{m-1}\na_y \frac{f^K_t}{\mu_{\beta}^K}\cdot \na_x^m\frac{f^K_t}{\mu_{\beta}^K} \dd \mu_{\beta}^K$:
\begin{multline*}
(\partial_tP^m)(t,\beta_t) \leqslant\int_{M_K\times M_K} -\left(\frac{1}{2}\mathbbm{1}_{\mathcal{M}} -\theta\mathbbm{1}_{\mathcal{M}^c}\right) |\na_x^{m}h^K_t|^2 \\ +\theta (1+\gamma_t)\left( \sum_{k=1}^{m-2} |\na_x^kh^K_t|^2  +(|\na\sigma|^2 + \sigma) |\na_x^{m+1}h^K_t|^2 + \sum_{k=1}^{m-1} |\na_x^k\na_yh^K_t|^2 \right) \dd\mu_{\beta_t}^K.
\end{multline*}
\end{lemma}

\begin{proof}
The derivative of $P^m$ is given by :
\[ (\partial_tP^m)(t,\beta_t) =  \int_{M_K\times M_K}\left( \na_x^{m-1}\na_yL_t^{K,*}h^K_t\cdot \na_x^mh^K_t + \na_x^{m-1}\na_yh^K_t\cdot \na_x^mL_t^{K,*}h^K_t \right)\dd \mu_{\beta_t}^K.\]
We then have to study the two terms in the integral separately at first, by using commutators. For any smooth h :
\begin{multline*}
\na_x^{m-1}\na_yL_t^{K,*}h\cdot \na_x^mh = \na_x^{m-1}\left[\na_y,L_t^{K,*}\right]h\cdot \na_x^mh \\+ \sum_{k=1}^{m-1} \na_x^{k-1} \left[\na_x,L_t^{K,*}\right]\na_x^{m-k}\na_yh\cdot \na_x^mh + L_t^{K,*}\na_x^{m-1}\na_yh\cdot \na_x^mh \\ = \na_x^{m-1}\left( -\na W\na_x- \na\sigma\na U^K.\na_x -\gamma_t\na W\na_y+\beta_t\na\sigma\Delta_x \right)h\cdot \na_x^mh \\  + \sum_{k=1}^{m-1} \na_x^{k-1} \left( \na^2U^K\na_y-\sigma\na^2U^K\na_x \right)\na_x^{m-k}\na_yh\cdot \na_x^mh + L_t^{K,*}\na_x^{m-1}\na_yh\cdot \na_x^mh \\ \leqslant -\left(\frac{1}{2}\mathbbm{1}_{\mathcal{M}} -\theta\mathbbm{1}_{\mathcal{M}^c}\right) |\na_x^{m}h|^2 + \theta\sum_{k=1}^{m-1}|\na_x^kh|^2 + \theta(1+\gamma_t)\sum_{k=1}^{m-1}|\na_x^{k}\na_yh|^2 \\+ \theta|\na\sigma|^2|\na_x^{m+1}h|^2 + L_t^{K,*}\na_x^{m-1}\na_yh\cdot \na_x^mh
\end{multline*}
where we used that $|\na\sigma\na U^K|\leqslant \frac{1}{2}$. Similarly, we have :
\begin{equation*}
\na_x^{m-1}\na_yh\cdot \na_x^mL_t^{K,*}h \leqslant \theta\sum_{k=1}^{m-1}|\na_x^{k}\na_yh|^2 \\+\theta\sum_{k=1}^{m-1}|\na_x^kh|^2 + \na_x^{m-1}\na_yh\cdot L_t^{K,*}\na_x^mh.
\end{equation*}
We conclude with the fact that :
\begin{multline*} 
\int_{M_K\times M_K} L_t^*\na_x^{m-1}\na_yh^K_t\cdot \na_x^mh^K_t + \na_x^{m-1}\na_yh^K_t\cdot L^*_t\na_x^mh^K_t \dd\mu^K_{\beta_t} \\= \int_{M_K\times M_K} \Gamma_t(\na_x^{m-1}\na_yh^K_t,\na_x^mh^K_t)\dd\mu^K_{\beta_t} \\ \leqslant  \int_{M_K\times M_k} \theta(1+\gamma_t)\left(|\na_x^m\na_yh_t^K|^2 + |\na_x^{m-1}\na_xh_t^K|^2 + |\na_x^m\na_y|^2\right) + \beta_t^{-1} \sigma |\na_x^{k+1}h|^2 \dd\mu_{\beta_t}^K.
\end{multline*}
\end{proof}

In order to state the main lemma, we introduce the following set :\lj{
\[\mathcal P = \left\{ g:\R_+\mapsto\R_+ \text{ smooth, }\exists a,C>0 \text{ s.t. } \beta^{-a}/C\leqslant g(\beta),|g'(\beta)| \leqslant C\beta^a \right\},\]where $\mathcal P$ stands for polynomial.}
The main step in order to prove an analogous to Proposition~\ref{evol0} in higher Sobolev norms is then  the following dissipation result.

\begin{lemma}\label{evol2}
For all $m\geqslant 1$, there exist $q_m$, $r_m$, $(\omega_{i,m})_{i\leqslant m}$ and $\omega_m$ some functions of $\beta$,  all in $\mathcal P$, such that if $f_0^K\in\mathcal{C}^{\infty}(M_K\times M_K)$ then, for all $t\geqslant0$, \[\|h^K_t-1\|_{H^m\left(\mu_{\beta_t}^K\right)}\leqslant q_m(\beta_t)\tilde{N}_m(t)\] and :
\begin{multline}
(\partial_tN_m)(t,\beta_t) \\ \leqslant -r_m(\beta_t)\int_{M_K\times M_K} \left(\sigma|\na_x^{m+1}h^K_t|^2 + \sum_{i=0}^m |\na_x^{i}\na_y^{m+1-i}h^K_t|^2 + \sum_{1\leqslant i+j\leqslant m}|\na_x^{i}\na_y^{j}h_t^K|^2 \right)\dd\mu_{\beta_t}^K. \label{eq:dissipHm}
\end{multline}
\end{lemma}

\begin{proof}
The proof is by induction. In fact we start at $m=0$, setting simply
\[N_0(t,\beta)=\int_{M_K\times M_K} \left( \frac{f^K_t}{\mu_{\beta}^K}-1 \right)^2 \dd\mu^K_{\beta}.\]
Hence,
\begin{align*}
(\partial_tN_0)(t,\beta_t) &= -\int_{M_K\times M_K} \Gamma_t\left( \frac{f^K_t}{\mu_{\beta_t}^K} \right) \dd\mu_{\beta_t}^K  \\&=-\int_{M_K\times M_K} \beta_t^{-1}\left(\gamma_t|\na_yh_t^K|^2 + \sigma|\na_xh_t^K|^2\right) \dd\mu_{\beta_t}^K
\end{align*}
which is the result for $m=0$ and $r_0(\beta_t)=\beta_t^{-1}\min(1,\kappa)$. One could also have initialized the induction for $m=1$ as in Proposition~\ref{evol0}.
Now, let's fix $m\in\mathbb{N}$ and suppose we have the result for $m-1$ :
\begin{multline*}
(\partial_tN_{m-1})(t,\beta_t) \leqslant \\-r_{m-1}(\beta_t)\int_{M_K\times M_K} \left(\sigma|\na_x^{m}h^K_t|^2 + \sum_{i=0}^{m-1} |\na_x^{i}\na_y^{m-i}h^K_t|^2 + \sum_{i+j\leqslant m-1}|\na_x^{i}\na_y^{j}h_t^K|^2 \right)\dd\mu_{\beta_t}^K.
\end{multline*}
We set
\[ N_m = N_{m-1} + \int_{M_K \times M_K} \left(\sum_{i=0}^m \omega_{i,m}(\beta)\left|\na_x^i\na_y^{k-i}\frac{f^K_t}{\mu_{\beta}^K}\right |^2 + \omega_m(\beta)\na_x^{m-1}\na_y\frac{f^K_t}{\mu_{\beta}^K}\cdot\na_x^m\frac{f^K_t}{\mu_{\beta}^K} \right)\dd\mu_{\beta}^K \] 
with weights to be determined later on, so that 
\[ (\partial_tN_m)(t,\beta_t)=(\partial_tN_{m-1})(t,\beta_t) - \int_{M_K \times M_K} \sum_{i=0}^m \omega_{i,m}(\beta) \Gamma_{|\na_x^i\na_y^{k-i}\cdot|^2}(h^K_t)\dd\mu_{\beta_t}^K +\omega_m(\beta) (\partial_tP^m)(t,\beta_t) .\]
In this equality, using the two previous lemmas, the terms of order $m+1$ are bounded by: 
\begin{multline*}
\int_{M_K\times M_K} \left(\omega_{m-1,m}|\na \sigma|^2+\omega_m(\sigma + |\na \sigma|^2)\right)|\na_x^{m+1}h_t^K|^2\\ - \beta_t^{-1}\sum_{i=0}^m \omega_{i,m}(\beta_t)\left( \gamma_t|\na_x^i\na_y^{m+1-i}h_t^K|^2 + \sigma|\na_x^{i+1}\na_y^{m-i}h_t^K|^2  \right) \dd\mu_{\beta_t}^K   
\end{multline*}
In order to get \eqref{eq:dissipHm}, we would like this to be less than
\[-r_m(\beta_t)\int_{M_K\times M_K} \po\sigma|\na_x^{m+1}h_t^K|^2 + \sum_{i=0}^m|\na_x^i\na_y^{m+1-i}h^K_t|^2 \pf \dd\mu_{\beta_t}^K.\]
Using that $|\na\sigma|^2\leqslant \sigma$, this is indeed the case if we impose the conditions $\omega_{m,m}(\beta) \geqslant 4\beta\omega_m(\beta) +\beta \omega_{m_1,m}$ and $\omega_{i,m}(\beta) \geqslant r_m \beta \max(\kappa^{-1},\frac{1}{2})$. Next, the terms of order at most $m-1$ are bounded by :
\[\int_{M_K \times M_K} \theta(1+\gamma_t)\max_{0\leqslant i\leqslant m}\omega_{i,m}\sum_{i+j\leqslant m-1}|\na_x^i\na_y^jh^K_t|^2 - r_{m-1}\sum_{i+j\leqslant m-1}|\na_x^i\na_y^jh^K_t|^2 \dd\mu_{\beta_t}^K \]
which means, similarly, in order to get \eqref{eq:dissipHm}, we want to impose the  conditions $\max_{0\leqslant i\leqslant m}\omega_{i,m} \theta(1+\gamma_t)\leqslant  \frac{1}{2}r_{m-1}$ and $r_{m-1}\geqslant 2r_{m}$. Finally, the terms of order exactly $m$ are bounded by
\begin{multline*} -\int_{M_K \times M_K}\left( r_{m-1}\sum_{i=1}^{m-1} |\na_x^i\na_y^{m-i}h^K_t|^2 + \left(r_{m-1}\sigma + \omega_m\frac{1}{2}\mathbbm{1}_{\mathcal{M}} - \omega_m\theta\mathbbm{1}_{\mathcal{M}^c}\right)|\na_x^mh^K_t|^2\right. \\\left.- \max_i\omega_{i,m}\theta(1+\gamma_t)\sum_{i=0}^m|\na_x^i\na_y^{m-i}h_t^K|^2\right) \dd\mu_{\beta_t}^K 
\end{multline*}
which leads to the conditions: $r_m\leqslant \frac{1}{2}\omega_m$, $\omega_m\theta \leqslant \frac{1}{2}r_{m-1}\sigma_*$ and $\frac{1}{2}r_{m-1}\sigma_*\geqslant r_{m}$.

Now, fix the $\omega_{i,m}$ to be equal to $\frac{r_{m-1}}{4\theta(1+L\beta)}$ except $\omega_{m,m}$ equal to $\frac{r_{m-1}}{2\theta(1+L\beta)}$, then $\omega_m=\min (\frac{\sigma_*r_{m-1}}{2\theta},\frac{\omega_{m,m}}{4\beta},\omega_{m,m})$, and finally set $r_m = \min(\frac{r_{m-1}}{2},\frac{r_{m-1}\sigma_*}{2},\frac{\omega_m}{2},\beta^{-1}\min(\kappa,2)\omega_{i,m})$. These choices ensures that all the conditions in the computations above are met, which means that \eqref{eq:dissipHm} holds. Moreover, all these functions are in $\mathcal P$, which concludes.
\end{proof}

Similarly to Proposition~\ref{poin} but now in the fully (position and velocity) compact case, we have the following Poincar\'e inequality:
\begin{proposition}\label{prop_poin}
If $U^K:M_K\mapsto \R$ is a $\C^{\infty}$ function, $W$ is as constructed in Section~\ref{fullprocess}, and $\mu^K_{\beta}$   is the probability measure on $M_K\times M_K$ proportional to $e^{-\beta H_K}$, then there exists $\lambda:\R_+\mapsto\R_+$ such that for all $f\in\C^{\infty}(M_K\times M_K)$ :
\[ \lambda(\beta)\int_{M_K \times M_K} \left(f-\int_{M_K\times M_K} f\dd\mu_{\beta}^K \right)^2\dd\mu_{\beta} \leqslant \int_{M_K\times M_K}|\nabla f|^2 d\mu^K_{\beta} \]
and
\[ \lim_{\beta\rightarrow\infty}\frac{1}{\beta}\ln(\lambda(\beta)) = -c^*(U^K) \]
where $c^*(U^K)$ is defined as $c^*$ with $U$ replaced by $U^K$.
\end{proposition}
\begin{proof} Using that $c^*(H_K)=c^*(U^K)$ since $W$ has only one minimum, this is \cite[Theorem 1.14]{HoStKu}.  
\end{proof}

We are now ready to prove  Proposition~\ref{evol1}. 

\begin{proof}[Proof of Proposition~\ref{evol1}]
Let $m=d+1$. We first show that there exist some constants $C_1,C_2>0$ such that $\tilde{N}_m(t)\leqslant C_1e^{-C_2t^{\alpha}}$ for all $t\geqslant $0, where $\alpha = \frac{c-c^*}{2c}$. Indeed, we have :
\[ \tilde{N}_m'(t) = \partial_tN_m(t,\beta_t) + \beta'_t\partial_{\beta}N_m(t,\beta_t) .\]
From Lemma~\ref{evol2} and Proposition~\ref{prop_poin}, 
\[\partial_tN_m \leqslant -\frac{\tilde \lambda(\beta_t)r_m}{2}N_m.\]
where $\tilde \lambda$ satisfies $\lim_{\beta\rightarrow\infty}\frac{1}{\beta}\ln(\tilde \lambda(\beta)) = -c^*(U^K)$ and $r_m\geqslant \beta^{-a}/C$ for some $a>0$. Hence, we have some constant $C>0$ such that $\lambda(\beta_t)r_m/2\geqslant Ce^{-\frac{c+c^*}{2}\beta_t} = C(e^{c\beta_0}+t)^{-1+\alpha}$.
On the other hand, from the conditions on the $\omega$'s and since $\partial_\beta e^{\beta H_K}\leqslant \|H_K\|_{\infty}e^{\beta H}$, there exist some $a,C'>0$ such that $\partial_{\beta}N_m\leqslant C' \beta_t^a N_m $ yielding $\beta'_t\partial_{\beta}N_m \leqslant C'(1+t)^{-1+\frac{\alpha}{2}}N_m$ for some other constant $C'>0$, and thus :
\[ \tilde{N}_{m}' \leqslant \left(-\frac{C}{(1+t)^{1-\alpha}} + \frac{C'}{(1+t)^{1-\frac{\alpha}{2}}}\right) \tilde{N}_m ,\]
yielding our first claim.

From the Sobolev embedding, there exists $C>0$ such that for any smooth $h$ on $M_K \times M_K$, we have $\|h\|_{L^{\infty}}\leqslant C\|h\|_{H^m\left(\mu_{\beta_0}^K\right)}$. Then we have :
\begin{multline*}
\|h\|_{L^{\infty}} \leqslant 1 + \|h-1\|_{L^{\infty}} \leqslant 1+C\|h-1\|_{H^m\left(\mu_{\beta_0}^K\right)}  \leqslant 1+Ce^{\beta_t\|H\|_{\infty}}\|h-1\|_{H^m\left(\mu_{\beta_t}^K\right)}. 
\end{multline*}
Applying this with $h=h_t^K$ and using that  $\|h^K_t-1\|_{H^m(\mu_{\beta_t}^K)}\leqslant C\beta_t^a\tilde{N}_m(t)$ for some $a,C>0$, we get constants $C,b>0$ such that :
\begin{equation*}
    \|h^K_t\|_{L^{\infty}} \leqslant 1+ C(1+t)^{b}e^{-C_2 t^{\alpha}},
\end{equation*}
yielding the result.
\end{proof}

\section{Faster than logarithmic schedules}\label{SecFast}

\begin{proof}[Proof of Theorem~\ref{thmFast}]
Some parts of the proof follow the proof of Theorem~\ref{non}, to which we refer for details, focusing  on the new arguments in the present settings.

Up to a translation we assume without loss of generality that $x_*=0$.  

As in the proof of Theorem~\ref{non}, it is in fact sufficient to prove that for all $\delta>0$
\[\mathbb P_{f_0}\po |Z_t|\leqslant \min(\delta,\sqrt{\varepsilon_t})\ \forall t\geqslant t_0\pf >0\]
for a given fixed initial condition $f_0$ and for $t_0$ large enough and then use a controlability argument based on Lemma~\ref{acces} and the Markov property to get the claimed result. Besides, it is sufficient to prove this result for $\delta$ small enough and, since the times in $[0,t_0]$ are treated with the controlability argument, focusing on the times larger than $t_0$, we are interested in an event under which the process stays in a small ball around $0$, and we can thus modify $U$ oustside such a ball without modifying the result.  

The Jacobian matrix of the drift $b(z)=(y,-\na U(x)-\gamma y)$ of the process at $z=(x,y)$ is
\[J(z) = \begin{pmatrix}
0 &  I \\ - \na^2 U(x) & -\gamma I
\end{pmatrix}\,.\]
\lj{Since $\na^2 U(0)$ is positive semi-definite, a simple calculation shows that} the eigenvalues of $J(0)$ all have a negative real part, and thus there exist a positive definite symmetric matrix $M$ of size $2d$ and $r>0$ such that
\[u\cdot M J(0) u \leqslant - 2r u \cdot M u\]
for all $u\in\R^{2d}$, see e.g. \cite[Lemma 4.3]{ArnoldErb}. Write $\|u\|_M = \sqrt{u\cdot Mu}$ and let $\delta>0$ be sufficiently small so that 
\begin{equation}\label{eq:contraction}
u\cdot M J(z) u \leqslant - r \|u\|_M^2    
\end{equation}
for all $u\in\R^{2d}$ and all $z\in\R^{2d}$ with $\|z\|_M \leqslant \delta$. As discussed above, up to a modification of $U$ outside the ball $\mathcal B(0,\delta)$, without loss of generality we can modify the potential $U$ outside this ball and assume that in fact \eqref{eq:contraction} holds for all $z\in\R^{2d}$.  Similarly we assume that $\na^3 U$ is bounded.

 Using that $b(0)=0$, we get that for all $z\in\R^{2d}$
\begin{equation}\label{eq:fastcontraction}
z\cdot M b(z) =z\cdot M \po b(z)-b(0)\pf = \int_0^1 z\cdot M J(pz)z \dd p \leqslant -r\|z\|_M^2\,. 
\end{equation}

Let $Z$ and $\tilde Z$ be two solutions of \eqref{eq} with the same initial condition and driven by the same Brownian motion, but with two different potentials, $Z$ being associated to $U$ and $\tilde Z$ to $\tilde U(z)=z\cdot \na^2 U(0) z/2$. In particular, $\tilde Z$ is a Gaussian process. Then
\begin{align*}
\dd \|Z_t-\tilde Z_t\|_M^2 &\leqslant -2r\|Z_t-\tilde Z_t\|_M^2\dd t + 2\|Z_t-\tilde Z_t\|_M |M^{1/2}||\na U(\tilde Z_t)-\na^2 U(0)\tilde Z_t|\dd t \\
& \leqslant  -r\|Z_t-\tilde Z_t\|_M^2\dd t + \frac1{4r}\|\na^3 U\|_\infty^2 |M| |\tilde Z_t|^4\dd t\,.
\end{align*}
Let $(\alpha_t)_{t\geqslant 0}$ be a positive non-increasing function, vanishing at infinity, to be chosen later on. Then the event $\mathcal G = \{|\tilde Z_t| \leqslant \alpha_t\ \forall t\geqslant 0\}$ implies that for all $t\geqslant 0$
\[|Z_t|\leqslant |\tilde Z_t| + |M^{-1/2}|\|Z_t-\tilde Z_t\|_M \leqslant \alpha_t +  C \po\int_0^t \alpha_s^2 e^{r(s-t)}\dd s \pf^{1/2} := \tilde \alpha_t\]
with $C^2= |M||M^{-1}|\|\na^3 U\|_\infty/(2\sqrt r)$. Notice that $\tilde\alpha_t$ vanishes at infinity, more precisely $\tilde \alpha_t \leqslant \alpha_t + C(\alpha_{3t/4}/\sqrt{r} + \alpha_0 \sqrt{t} e^{-rt/8})$.
 As a consequence, it only remains to prove that $\mathcal G$ has a positive probability for some suitable function $t\mapsto \alpha_t$. Again, following the proof of Theorem~\ref{non}, we see that it is sufficient to prove that $t\mapsto \mathbb P(|\tilde Z_t|\geqslant \alpha_t)$ is integrable for a given fixed initial condition. Since $\tilde Z$ is a Gaussian process, we simply have to control its second moment.  We chose an initial condition such that $\mathbb E(\tilde Z_0)=0$, so that $\mathbb E(\tilde Z_t)=0$ for all $t\geqslant 0$.

Remark that the drift of $\tilde Z_t$ also satisfies \eqref{eq:fastcontraction}, since its Jacobian matrix at any $z\in\R^{2d}$ is equal to $J(0)$. Thus we get, for all $t\geqslant 0$,
\[\partial_t \mathbb E\po \|\tilde Z_t\|_M^2 \pf \leqslant - r \mathbb E\po \|\tilde Z_t\|_M^2 \pf  + \frac{d|M|}{\beta_t} ,\]
so that, writing $C'=|M||M^{-1}|( \mathbb E|\tilde Z_0|^2+d)$,
 \[   \mathbb E\po |\tilde Z_t|^2 \pf  \leqslant |M^{-1}|\mathbb E\po \|\tilde Z_t\|_M^2 \pf  \leqslant  C' \po e^{-rt} + \int_0^t \frac{e^{r(s-t)}}{\beta_s}\dd s \pf  
    := C'\kappa_t\,,\]
 Using that the law of $\tilde Z_t$ is Gaussian with zero mean, we get that there exist $K,h>0$ such that for all $t\geqslant 0$
\[\mathbb P(|\tilde Z_t| \geqslant \alpha_t)  \leqslant K \exp\po - h \frac{\alpha_t^2}{\kappa_t}\pf.\]  
This is integrable in time if we chose $\alpha_t^2=\sup_{s\geqslant t}2\ln(s)\kappa_s /h$ (which is non-increasing). Since $\ln(t)=o(\beta_t)$ and $\kappa_t \leqslant \sup_{s\geqslant 2t/3}1/(r\beta_{s})+e^{-rt/3}(1+t\sup_{u\geqslant 0}1/\beta_u)$, we get that $\kappa_t = o( 1/\ln(t))$, in other words $\alpha_t \rightarrow 0$. We conclude by using the previous bounds on $\tilde \alpha_t$ and   $\kappa_t$ and the fact that
\[\sup_{s\geqslant 3t/4} \ln(s) \sup_{u\geqslant 2s/3} \beta_u^{-1} \leqslant \sup_{s\geqslant t/2} \ln(s) \beta_s^{-1}\]
to get the quantitative convergence speed $\varepsilon_t$ stated in the theorem.

\end{proof} 

\noindent\textbf{Remark:} in order to adapt the previous proof in a case where $\gamma_t$ is not constant, we can still find $M$ and $r$ such  that \eqref{eq:contraction} holds with the Jacobian matrix of the drift at time $t$, but they depend on $\gamma_t$, hence when differentiating $\|Z_t-\tilde Z_t\|_M^2$ there is an additional term involving $\partial_t M$ which has to be sufficiently small to be absorbed by the contraction at rate $r$. Moreover, $r$ scales as $\gamma_t$ when $\gamma_t\rightarrow 0$ and as $1/\gamma_t$ when $\gamma_t\rightarrow +\infty$, which means for the proof to be valid, $\gamma_t$ or $1/\gamma_t$ should not be too small depending on $\beta_t$.

\subsection*{Acknowledgements}

This work has been partially funded by the  French ANR grants EFI (ANR-17-CE40-0030) and SWIDIMS (ANR-20-CE40-0022). P. Monmarch\'e thanks Laurent Dietrich for his help in the proof of Lemma~\ref{acces}. \pierre{The authors thank Martin Chak for pointing out an error in an earlier version of Lemma~\ref{regu}}.

\bibliographystyle{plain}
\bibliography{Recuit_Arxiv.bib}

\end{document}